\documentclass{amsart}
\usepackage{amsmath}
\usepackage{amsfonts}
\usepackage{amstext}
\usepackage{amsbsy}
\usepackage{amsopn}
\usepackage{amsxtra}
\usepackage{upref}
\usepackage{amsthm}
\usepackage{amsmath}
\usepackage{amssymb}

\newtheorem{prop}{Proposition}[section]
\newtheorem{lema}[prop]{Lemma}
\newtheorem{rem}[prop]{Remark}
\newtheorem{defi}[prop]{Definition}
\newtheorem{teo}[prop]{Theorem}
\newtheorem{eje}[prop]{Example}

\renewcommand{\dim}{\text{dim}_{\rm{H}}}
\def\R{{\mathbb R}}
\def\N{{\mathbb N}}

\def\M{{\mathcal M}}

\def\C{{\mathcal C}}

\usepackage{euscript}

\DeclareMathSymbol{\varnothing}{\mathord}{AMSb}{"3F}
\renewcommand{\emptyset}{\varnothing}

\title[Multifractal analysis for quotients of Birkhoff sums]{Multifractal analysis for quotients of Birkhoff sums for countable Markov maps}
\date{\today}

\begin{thanks}
{G.I. was partially  supported by  the Center of Dynamical Systems and Related Fields c\'odigo ACT1103 and by Proyecto Fondecyt 1110040. The work on this project was started during T.J.'s visit to Santiago supported by Proyecto Mecesup-0711. We would also like to thank Lingmin Liao and Micha\l\text{ }Rams for helpful comments. Finally, we wish to thank the referee for all the interesting remarks and comments.}
\end{thanks}

\author{Godofredo Iommi} \address{Facultad de Matem\'aticas,
Pontificia Universidad Cat\'olica de Chile (PUC), Avenida Vicu\~na Mackenna 4860, Santiago, Chile}
\email{giommi@mat.puc.cl}
\urladdr{http://www.mat.puc.cl/\textasciitilde giommi/}
\author{Thomas Jordan} \address{The School of Mathematics, The University of Bristol, University Walk, Clifton, Bristol, BS8 1TW, UK}
\email{Thomas.Jordan@bristol.ac.uk}
\urladdr{http://www.maths.bris.ac.uk/~matmj}

\begin{document}

\begin{abstract}
This paper is devoted to study multifractal analysis of quotients of Birkhoff averages for countable Markov maps. We prove a variational principle for the Hausdorff dimension of the  level sets. Under certain assumptions we are able to show that the spectrum varies analytically in parts of its domain. We apply our results to show that the Birkhoff spectrum for the Manneville-Pomeau map can be discontinuous, showing the remarkable differences with the uniformly hyperbolic setting. We also obtain results describing the Birkhoff spectrum of suspension flows. Examples involving continued fractions are also given.
\end{abstract}

\maketitle

\section{Introduction}
Multifractal analysis is a branch of the dimension theory of dynamical systems. It typically involves decomposing the phase space into level sets where some local quantity takes a fixed value, say $a$. The standard problems are to find the Hausdorff dimension of these sets and to determine how the dimension varies with the parameter $a$. In hyperbolic dynamical systems there are several local quantities which can be studied in such a way. These quantities are dynamically defined, examples are local dimensions of Gibbs measures, Birkhoff averages of continuous functions and local entropies of Gibbs measures. In this setting we normally find two types of results, on the one hand variational principles are obtained to determine the dimension of the level sets and on the other thermodynamic formalism is used to prove that the spectra varies in an analytic way.  See \cite{ba} for an overview of some of these results.

In this note we consider the multifractal analysis for quotients of Birkhoff averages for a countable branch Markov map. In the case of expanding finite branch Markov maps on the interval the multifractal analysis for Birkhoff averages or quotients of Birkhoff averages of continuous functions is well understood, for example see \cite{bs1, flw, olsen, flp, c}. In particular, if the Markov map is expanding, $C^{1+\epsilon}$ and the continuous functions are H\"{o}lder the multifractal spectra vary analytically \cite{bs1}. However, it turns out that substantial differences can occur in the case where there are countably many branches and Birkhoff averages are studied, note that the space is no longer compact. In this setting, phase transitions may occur, for example see the work in \cite{fjlr, ij,KMS}. Nevertheless,  the dimension of sets in the multifractal decomposition can still generally be found by a conditional variational principle. Our aim is to extend these results to fairly general quotients of Birkhoff sums and with additional assumptions examine the smoothness of the multifractal spectra. Note that the multifractal analysis of quotients of Birkhoff sums has already been studied in \cite{ku} but with the assumption that the denominator is the Lyapunov exponent.

One motivation for our work is that studying quotients of Birkhoff averages for a countable branch Markov map can relate to studying Birkhoff averages for a finite branch non-uniformly expanding map. In \cite{jjop} multifractal analysis for Birkhoff on finite branch non-uniformly expanding maps was studied, while a conditional variational principle was obtained the question of how the spectra varied was not fully addressed. By tackling this problem via countable state expanding maps we are able to obtain new results for H\"{o}lder functions in this setting and to show that while in some regions the spectrum varies analytically it is also possible for it to have discontinuities. Another motivation is to study related problems for suspension flows where Birkhoff averages on the flow correspond to quotients of Birkhoff averages for the base map.

Let us be more precise and define the dynamical systems under consideration. Denote by $I=[0,1]$  the unit interval. As in \cite{ij} we consider the  class of EMR (expanding-Markov-Renyi) interval maps.
\begin{defi} \label{maps} Let $\{ I_i \}_{i\in\N}$ be a countable collection of closed intervals where where $\textrm{int}(I_i)\cap \textrm{int}(I_j)=\emptyset$ for $i,j\in\N$ with $i\neq j$ and $I_i \subset I$ for every $i \in \mathbb{N}$.
A map $T:\cup_{n=1}^{\infty}I_n \to I$ is an EMR map, if the following properties are satisfied
\begin{enumerate}
\item The intervals in the partition are ordered in the sense that for every $i \in \N$ we have $\sup \{x : x \in I_{i}\} \leq \inf \{x : x \in I_{i+1}\}$. Moreover,  zero is the unique  accumulation point  of the set of endpoints of $\{I_i\}$.
\item The map is $C^2$ on $\cup_{i=1}^{\infty} \textrm{int }I_i$.
\item There exists $\xi >1$ and $N\in\N$ such that for every $x \in \cup_{i=1}^{\infty} I_i$ and $n\geq N$
we have $|(T^n)'(x)|>\xi^n$.
\item The map $T$ is Markov and it can be coded by a full-shift on a countable alphabet.
\item The map satisfies the Renyi condition, that is, there exists a positive number $K>0$
such that
\[ \sup_{n \in \N} \sup_{x,y,z \in I_n} \frac{|T''(x)|}{|T'(y)| |T'(z)|} \leq K. \]
\end{enumerate}
\end{defi}

The \emph{repeller} of such a map is defined by
\[\Lambda:=\left\{x \in \cup_{i=1}^{\infty} I_i: T^n(x) \textrm{ is well defined for every } n \in \mathbb{N} \right\}.\]
The Markov structure assumed for EMR maps $T$, allows for a good symbolic representation (see Section \ref{st}).
\begin{eje} \label{ga}
The Gauss map $G:(0,1] \to (0,1]$ defined by
\[G(x)= \frac{1}{x} -\Big[ \frac{1}{x} \Big],\]
where $[ \cdot]$ is the integer part, is a standard example of an EMR map.
\end{eje}

%

For $\phi\in\mathcal{R}$ and $\psi\in\mathcal{R}_{\eta}$ (see Definition \ref{pot} for a precise definition of the class of potentials $\mathcal{R}$ and $\mathcal{R}_{\eta}$ )  we will denote
\begin{eqnarray*}
\alpha_m=\inf \left\{ \lim_{n \to \infty} \frac{\sum_{i=0}^{n-1} \phi (T^i x)}{\sum_{i=0}^{n-1} \psi (T^i x)}: x \in  \Lambda \right\} \textrm{ and } &\\
\alpha_M=\sup \left\{ \lim_{n \to \infty} \frac{\sum_{i=0}^{n-1} \phi (T^i x)}{\sum_{i=0}^{n-1} \psi (T^i x)}: x \in  \Lambda \right\}.
\end{eqnarray*}
Throughout the paper we will assume that $\alpha_m\neq\alpha_M$ since otherwise our results become trivial.
Note that, since the space $\Lambda$ is not compact,  it is possible for $\alpha_M$ to be  infinity.
For $\alpha \in [\alpha_m,\alpha_M]$ we define the level set  of points having Birkhoff ratio equal to $\alpha$ by
\begin{equation*}
J(\alpha)= \left\{x \in  \Lambda :    \lim_{n \to \infty} \frac{\sum_{i=0}^{n-1} \phi (T^i x)}{\sum_{i=0}^{n-1} \psi (T^i x)}
 = \alpha \right\}.
\end{equation*}
Note that these sets induce the so called \emph{multifractal decomposition} of the repeller,
\begin{equation*}
\Lambda= \bigcup_{\alpha=\alpha_m}^{\alpha_M}J(\alpha) \text{ } \bigcup J',
\end{equation*}
where $J'$ is the \emph{irregular set} defined by,
\[J' = \left\{x \in  \Lambda :  \textrm{ the limit }
 \lim_{n \to \infty} \frac{\sum_{i=0}^{n-1} \phi (T^i x)}{\sum_{i=0}^{n-1} \psi (T^i x)}  \textrm { does not exist }  \right\}. \]
The \emph{multifractal spectrum} is the function that encodes this decomposition and it is defined by
\[b(\alpha)= \dim(J(\alpha)),\]
where $\dim(\cdot)$ denotes the Hausdorff dimension (see Subsection \ref{hausdorff}).
The form of our results depend upon the value of $\alpha$. We will split the open interval $(\alpha_m,\alpha_M)$ into two subsets. Firstly  we define
$$\underline{\alpha}=\inf\left\{\alpha \in \R:\text{ there exists }\{x_n\}_{n\in\N}\text{ where }\lim_{n\to\infty}x_n=0\text{ and } \lim_{n\to\infty}\frac{\phi(x_n)}{\psi(x_n)}=\alpha\right\}$$
and
$$\overline{\alpha}=\sup\left\{\alpha\in \R:\text{ there exists }\{x_n\}_{n\in\N}\text{ where }\lim_{n\to\infty}x_n=0\text{ and } \lim_{n\to\infty}\frac{\phi(x_n)}{\psi(x_n)}=\alpha\right\}.$$
We let
$E=[\underline{\alpha},\overline{\alpha}]$, $U=(\alpha_m,\alpha_M)\backslash E$ and $\tilde{\mathcal{M}}(T)$ be the space of all $T$-invariant probability measures for which $\psi$ and $\log |T'|$ are integrable.
We can now state our first result. In our first theorem we establish a variational principle for the dimension of level sets.

\begin{teo}\label{main}
If $\phi\in\mathcal{R}$ and $\psi\in\mathcal{R}_{\eta}$ then for all $\alpha\in U$
$$b(\alpha):=\dim J(\alpha)=\sup_{\mu\in \tilde{\mathcal{M}}(T)}\left\{\frac{h(\mu)}{\lambda(\mu)}: \frac{\int\phi\text{d}\mu}{\int\psi\text{d}\mu}=\alpha\right\}$$
and for $\alpha\in E\cap (\alpha_m,\alpha_M)$
$$b(\alpha):=\dim J(\alpha)=\lim_{\epsilon\to 0}\sup_{\mu\in \tilde{\mathcal{M}}(T)}\left\{\frac{h(\mu)}{\lambda(\mu)}:\frac{\int\phi\text{d}\mu}{\int\psi\text{d}\mu}\in (\alpha-\epsilon,\alpha+\epsilon)\right\}.$$
\end{teo}
Here $\lambda(\mu)$ is the Lyapunov exponent and $\mathcal{M}(T)$ is the space of $T-$invariant measures (see section \ref{st} for precise definitions).
A particular case we will be interested is when the function $\psi$ satisfies that $\lim_{x\to 0}\frac{\psi(x)}{\log |T'(x)|}=\infty$ and also $\phi$ and $\psi$ are such that $\alpha_M<\infty$. With these additional assumptions we are able to say more about the smoothness of the function $\alpha\to\dim J(\alpha)$.

\begin{teo}\label{analytic}
Let $\phi\in\mathcal{R}$ and $\psi\in\mathcal{R}_{\eta}$ be such that
\[\lim_{x\to 0}\frac{\psi(x)}{\log |T'(x)|}=\infty,\]
and that $\alpha_M<\infty$. There then exists three pairwise disjoint intervals $J_1,J_2,J_3$ such that
\begin{enumerate}
\item
$J_1\cup J_2\cup J_3=(\alpha_m,\alpha_M)$,
\item
$J_1\leq J_2\leq  J_3$
\item
The function $\alpha\to\dim X_{\alpha}$ is analytic on $J_1$ and $J_3$
\item
For $\alpha\in J_2$, $\dim X_{\alpha}=\dim\Lambda$
\item
It is possible that $J_1=\emptyset$, $J_3=\emptyset$ or that $J_2$ is a single point.
\end{enumerate}
\end{teo}

The rest of the paper is laid out as follows. In section 2 we give some preliminary results necessary for the rest of paper. In sections 3 and 4 respectively the upper and lower bounds for Theorem \ref{main} are proved. In section 5 Theorem \ref{analytic} is proved, section 6 shows there may be examples with discontinuities in the spectrum and gives applications to non-uniformly expanding maps. Section 7 looks at the case of suspension flows and finally in section 8 we provide examples coming form number theory.

\section{Preliminaries}
This section is devoted to provide the necessary tools and definitions that will be used in the rest of the paper.

\subsection{Thermodynamic Formalism for EMR maps.} \label{st}

In order to define the thermodynamic quantities and to establish their properties for an EMR map we will make use of the analogous theory developed  at a symbolic level. Let $\mathbb{N}$ be the countable alphabet, the \emph{full-shift} is the pair $(\Sigma, \sigma)$ where $\Sigma = \left\{ (x_i)_{i \ge 1} : x_i \in \mathbb{N}\right\},$
and $\sigma: \Sigma \to \Sigma$ is the \emph{shift} map  defined by $\sigma(x_1x_2 \cdots)=(x_2 x_3\cdots)$. We equip $\Sigma$ with the topology generated by the cylinders sets
\[ C_{i_1 \cdots i_n}= \{x \in \Sigma : x_j=i_j \text{ for } 1 \le j \le n \}.\]
The Markov structure assumed in the definition of EMR map implies that there exists a continuous map, the natural projection,  $\pi :\Sigma \to \Lambda$ such that $\pi \circ \sigma = T  \circ  \pi$.  Moreover the map $\pi:\Sigma \to \Lambda \setminus \bigcup_{n \in \N} T^{-n} E$ is surjective and injective except on at most a countable set of points.  Denote by $I(i_1, \dots i_n)= \pi (C_{i_1 \dots i_n})$ the cylinder of length $n$ for $T$. For a function $f\in\Sigma$  and $n\geq 1$ we will define the $n-$\emph{variations} of $f$ by
$$\text{var}_n(f)=\sup_{(i_1,\ldots,i_n)\in\N^n}{\sup_{x,y\in C_{i_1 \cdots i_n}}}|f(x)-f(y)|$$
and say that $f$ is locally H\"{o}lder if there exists $0<\gamma<1$ and $A>0$ such that for all $n \geq 1$ we have $\text{var}_n(f) \leq A \gamma^n$.  We now define the main object in thermodynamic formalism,
\begin{defi}
The \emph{topological pressure} of a potential $\phi:\Lambda\to\R$  such that $\phi\circ\pi$  is locally H\"{o}lder is defined by
 \[P_T(\phi)=  \sup \left\{ h(\mu) +\int  \phi \, d \mu : -\int  \phi \, d
\mu < \infty \textrm{ and } \mu\in \mathcal{M}_T  \right\}, \]
where $ \mathcal{M}_T$ denotes the space of $T-$invariant probability measures. A measure attaining the supremum is called an \emph{equilibrium measure} for $\phi$.\end{defi}
The following definition of pressure (at a symbolic level) is due to Mauldin and Urba\'nski \cite{mu},
\begin{defi}
Let $\phi: \Sigma \to \mathbb{R}$ be a potential of summable variations, the \emph{pressure} of $\phi$ is defined by
\begin{equation}
P_{\sigma}(\phi) = \lim_{n \to \infty} \frac{1}{n} \log \sum_{\sigma^n(x)=x} \exp \left( \sum_{i=0}^{n-1} \phi(\sigma^i x) \right).
\end{equation}
\end{defi}
The above limit always exits,  but it can be infinity. The next proposition relates these two notions and allows us to translate results obtained by Mauldin and Urba\'nski \cite{mu, mubook}
and by Sarig \cite{sa1, sar, sa3} to our setting. For $n\in\N$ we will denote
$$\Sigma_n=\{x\in\Sigma:x_i\leq n\}$$
and $\Lambda_n=\pi(\Sigma_n)$. Note that $\Lambda_n$ is a $T$-invariant set.
\begin{prop}\label{tf_maps}
Let $T$ be an EMR map. If $\phi:\Lambda\to\R$  such that  $\Phi=\phi\circ\pi$  is locally H\"older then
\begin{enumerate}
\item
$$P_T(\phi)=P_{\sigma}(\Phi).$$
\item \emph{(Approximation property.)}
\begin{equation*}
P( \phi) = \sup \{ P_{\sigma|K}( \Phi) : K \subset (0,1] : K \ne \emptyset \text{ compact and } \sigma\text{-invariant}  \},
\end{equation*}
where $P_{\sigma|K}( \phi)$ is the classical topological pressure on $K$ (for a precise definition see \cite[Chapter 9]{wa}).
In particular
$$P(\phi)=\sup_{n\in\N}\{P_{T|\Lambda_n}(\phi)\}.$$
\item \emph{(Regularity.)}
If $P(\phi)< \infty$ then there exists a \emph{critical value} $t^{*} \in (0, 1]$ such that for every $t <t^{*}$we have that $P(t \phi)= \infty$ and for every $t > t^{*}$we have that $P(t \phi)< \infty$. Moreover, if $t>t^{*}$ then
the function $t \to P(t \phi)$  is real analytic, strictly convex and every potential $t \phi$ has an unique equilibrium measure. Moreover the function $t \to P_{T|\Lambda_n}(t \phi)$ is analytic and convex for all $t\in\R$.
\end{enumerate}
\end{prop}

Note that if $T$ is an EMR map then the potential $\log |T'| \circ\pi$ is locally H\"older and $P(-\log|T'|) < \infty$. If $\mu \in \mathcal{M}_T$ then  the integral
\[\lambda(\mu):= \int \log |T'| \, d\mu, \]
will be called the \emph{Lyapunov exponent} of $\mu$. Of particular interest will be  the following classes of potentials,
\begin{defi}\label{pot}
We define the following collections of potentials
$$\mathcal{R}=\left\{\phi:\Lambda\to\R:\phi\text{ is uniformly bounded below and } \phi\circ\pi\text{ is locally H\"{o}lder}\right\}$$
and for $\eta>0$
$$\mathcal{R}_{\eta}=\left\{\phi \in \mathcal{R}: \text{for every } x\in\Lambda \text{ we have  }\phi(x)\geq \eta \right\}.$$
\end{defi}

\subsection{Hausdorff Dimension} \label{hausdorff}
In this subsection we recall basic definitions  from dimension theory. We refer to the books \cite{ba, fa} for further details. A countable collection of sets $\{U_i \}_{i\in N}$ is called a $\delta$-cover of $F \subset\R$ if $F\subset\bigcup_{i\in\N} U_i$, and  for every $i\in\N$ the sets $U_i$ have diameter $|U_i|$ at most $\delta$. Let $s>0$, we define
\[
 H^s_{\delta} (F) :=\inf \left\{ \sum_{i=1}^{\infty} |U_i|^s : \{U_i \}_i \text{ is a } \delta\text{-cover of } F \right\}
\]
and
\[ H^s(F):=  \lim_{\delta \to 0} H^s_{\delta} (F).\]
The \emph{Hausdorff dimension} of the set $F$ is defined by
\[
{\dim}(F) := \inf \left\{ s>0 : H^s(F) =0 \right\}.
\]
We will also define the \emph{Hausdorff dimension} of a probability measure $\mu$ by
\[
{\dim}(\mu) := \inf \left\{ \dim(Z): \mu(Z)=1 \right\}.
\]

\section{Proof of Upper bound of Theorem \ref{main}}
 Throughout this section we will let
let $\phi\in\mathcal{R}$ and $\psi\in\mathcal{R}_{\eta}$. We wish to prove that
\begin{equation}\label{ub1}
\dim J(\alpha)\leq\sup_{\mu\in \tilde{\mathcal{M}}(T)}\left\{\frac{h(\mu)}{\lambda(\mu)}:\frac{\int\phi\text{d}\mu}{\int\psi\text{d}\mu}=\alpha\right\}:=\delta(\alpha).
\end{equation}
However to prove this we will need that $\alpha\in U$. When $\alpha\notin U$, in general, we can only show that
\begin{equation}\label{ub2}
\dim J(\alpha)\leq\lim_{\epsilon\to 0}\sup_{\mu\in \tilde{\mathcal{M}}(T)}\left\{\frac{h(\mu)}{\lambda(\mu)}:\left|\frac{\int\phi\text{d}\mu}{\int\psi\text{d}\mu}-\alpha\right|\leq\epsilon\right\}.
\end{equation}
Our method will be to prove (\ref{ub2}) for all $\alpha\in (\alpha_m,\alpha_M)$ and then to deduce (\ref{ub1}) for $\alpha\in U$ by showing that $\delta$ is continuous in this region.
We first show the continuity of $\delta(\alpha)$ when $\alpha\in U$. To do this we need the following preparatory Lemma about the set $U$. This Lemma will also be used in the proof of Theorem \ref{analytic}. Denote by $B(\alpha,\gamma)$ the ball of center $\alpha$ and radius $\gamma$.

\begin{lema}\label{bounded}
For any $\alpha\in U$ there exists $\gamma>0$ and $C_1>0$ such that if $\mu \in \tilde{\mathcal{M}}_T$ with $\frac{\int\phi\text{d}\mu}{\int\psi\text{d}\mu}\in B(\alpha,\gamma)$ then $\int\psi\text{d}\mu\leq C_1$.
\end{lema}

\begin{proof}
We will let $\alpha\in U$ and assume that $\alpha<\underline{\alpha}$, since the case when $\alpha>\overline{\alpha}$ can be treated analogously. We let $\gamma=\frac{\underline{\alpha}-\alpha}{3}$ and note that by the definition of $U$ there exists $y' \in (0,1)$ such that if $x\leq y'$ then $\phi(x)\geq (\alpha+2\gamma)\psi(x)$.
We also let $C_2 \in \R$ be such that
\[ \max \left\{\sup_{x\geq y'}\{|\phi(x)|\},\sup_{x\geq y'}\{|\psi(x)|\} \right\} \leq C_2.\]
Assume that $\mu$ is a $T$-invariant probability measure such that
\begin{equation*}
 \int\psi\text{d}\mu<\infty \quad \text{ and } \quad \frac{\int \phi\text{d}\mu}{ \int \psi\text{d}\mu} \in B(\alpha,\gamma).
\end{equation*}
 We can estimate
$$(\alpha+\gamma)\int\psi\text{d}\mu\geq\int\phi\text{d}\mu\geq\int_{x\geq y'}\phi(x)\text{d}\mu+(\alpha+2\gamma)\int_{x\leq y'}\psi(x)\text{d}\mu$$
and rearrange to get
$$-\gamma \int_{x\leq y'}\psi(x)\text{d}\mu +(\alpha+\gamma)\int_{x\geq y'}\psi(x)\text{d}\mu -\int_{x\geq y'}\phi(x)\text{d}\mu \geq 0.$$
Thus, using the definition of $C_2$ we obtain
$$\gamma \int_{x\leq y'}\psi(x)\text{d}\mu \leq C_2|\alpha+\gamma+1|.$$
Therefore,
$$\int\psi\text{d}\mu\leq C_2\left(\frac{|\alpha+\gamma+1|}{\gamma}+1\right)$$
which completes the proof.
 \end{proof}

\begin{lema}\label{continuity}
The function $\delta:(\alpha_m,\alpha_M)\to \R$ defined by,
$$\delta(\alpha):=\sup_{\mu\in \tilde{\mathcal{M}}(T)}\left\{\frac{h(\mu)}{\lambda(\mu)}: \frac{\int\phi\text{d}\mu}{\int\psi\text{d}\mu}= \alpha \right\}$$
is continuous on $U$.
\end{lema}

\begin{proof}
Let $\alpha\in U$. We will show that
\begin{eqnarray*}
\liminf_{\epsilon\to 0}\inf\{\delta(\gamma):\gamma\in (\alpha-\epsilon,\alpha+\epsilon)\}\geq \delta(\alpha) \text{ and } &\\
\limsup_{\epsilon\to 0}\sup\{\delta(\gamma):\gamma\in (\alpha-\epsilon,\alpha+\epsilon)\}\leq \delta(\alpha).
\end{eqnarray*}
We can find two measures $\mu_1,\mu_2\in \tilde{\mathcal{M}}_T$ such that $\lambda(\mu_1),\int \psi \text{d}\mu_1,\int  \psi \text{d}\mu_2<\infty$ and
$$\frac{\int\phi\text{d}\mu_1}{\int\psi\text{d}\mu_1}<\alpha<\frac{\int\phi\text{d}\mu_2}{\int\psi\text{d}\mu_2}.$$
By Lemma \ref{bounded} there also exists a constant $K>0$ such that for any $\epsilon>0$ we can find a measure $\mu$ such that $\int \psi\text{d}\mu\leq K$,
$\frac{\int\phi\text{d}\mu_1}{\int\psi\text{d}\mu_1}=\alpha$ and $\frac{h(\mu)}{\lambda(\mu)}\geq\delta(\alpha)-\epsilon$. To prove that $\liminf_{\epsilon\to 0}\inf\{\delta(\gamma):\gamma\in (\alpha-\epsilon,\alpha+\epsilon)\}\geq \delta(\alpha)$ we simply need to consider the following convex families of measures $p\mu+(1-p)\mu_1$ and $p\mu+(1-p)\mu_2$, where $p\in (0,1)$.

To show that $\limsup_{\epsilon\to 0}\sup\{\delta(\gamma):\gamma\in (\alpha-\epsilon,\alpha+\epsilon)\}\leq \delta(\alpha)$ we consider a sequence of $T$-invariant measures $(\nu_n)_n$ such that
\begin{equation*}
\lim_{n\to \infty} \frac{h(\nu_n)}{\lambda(\nu_n)}\geq \delta(\alpha) \text{ and } \lim_{n\to\infty}\frac{\int\phi\text{d}\nu_n}{\int\psi\text{d}\nu_n}=\alpha.
\end{equation*}
By considering the appropriate convex combination with either $\mu_1$ or $\mu_2$ we can now find a sequence $(\eta_n)_n$ of $T$-invariant measures with
$\frac{\int\phi\text{d}\eta_n}{\int\psi\text{d}\eta_n}=\alpha$ and $\lim_{n\to \infty} \frac{h(\eta_n)}{\lambda(\eta_n)}\geq \delta(\alpha)$. The result immediately follows.
\end{proof}

Let $C>0$ be a constant such that
 \[ \max\left\{\sum_{k=1}^{\infty} \text{var}_k (\phi),\sum_{k=1}^{\infty} \text{var}_k(\psi),\sum_{k=1}^{\infty} \text{var}_k(\log |T'|) \right\}\leq  C. \]
We work in a similar way to \cite{ij}. Let  $S_k\phi(x):= \sum_{i=0}^{k-1}\phi(T^ix)$ and
$$J_{\alpha,N,\epsilon}:=\left\{x\in\Lambda:\frac{S_k\phi(x)}{S_k\psi(x)}\in (\alpha-\epsilon,\alpha+\epsilon)\text{ for every }k\geq N\right\}.$$
Note that for any $\epsilon>0$ we have
$$J_{\alpha}\subset\cup_{N=1}^{\infty}J(\alpha,N,\epsilon).$$
So we can obtain upper bounds of the dimension of the set $J(\alpha)$ by obtaining upper bounds for the dimension of $J(\alpha,N,\epsilon)$. For $k\geq N$  we will define covers by
$$\mathcal{C}_k=\{I(i_1,\ldots,i_k):I(i_1,\ldots,i_k)\cap J(\alpha,N,\epsilon)\neq\emptyset\}.$$
In \cite{ij} analogous covers were defined in Section 3. However, in that setting these covers had finite cardinality whereas here the cardinality can be infinite which will cause additional difficulties.
From now on $\alpha$ and $\epsilon>0$ will be fixed. We will let
$$t_k:=\inf\left\{t \in \R:\sum_{I(i_1,\ldots,i_k)\in \mathcal{C}_k}|I(i_1,\ldots,i_k)|^t\leq 1 \right\}.$$
A covering argument then gives that $\dim J(\alpha,N,\epsilon)\leq\limsup_{k\rightarrow\infty}t_k$. We wish to relate the values $t_k$ to $T$-invariant probability measures. We start with the following lemma.

\begin{lema}\label{simpest}
There exists $K\in\N$ such that for all $k\geq K$, $I(i_1,\ldots,i_k)$ and $x,y\in I(i_1,\ldots,i_k)$ we have
$$\frac{S_k\phi(x)}{S_k\psi(x)}-\frac{S_k\phi(y)}{S_k\psi(y)}\leq\epsilon.$$
\end{lema}

\begin{proof}
By the assumption that both $\psi$ and $\phi$ are of summable variations we know that for all $k\in\N$.
$$\frac{S_k\phi(x)-C}{S_k\psi(x)+C}\leq\frac{S_k\phi(y)}{S_k\psi(y)}\leq\frac{S_k\phi(x)+C}{S_k\psi(x)-C}.$$
However, we have by assumption that $S_k\psi(x)\geq k\eta$ for all $x\in\Lambda$. The result now immediately follows.
\end{proof}
We can now construct the measures we need.

\begin{lema}
If $t\in\R$ satisfies that
$$\sum_{I(i_1,\ldots,i_k)\in \mathcal{C}_k}|I(i_1,\ldots,i_k)|^t>1$$
then there exists a $T$-invariant probability measure $\mu_k$ such that
$$\frac{\int\phi\text{d}{\mu_k}}{\int\psi\text{d}\mu_k}\in (\alpha-2\epsilon,\alpha+2\epsilon)$$
and
$$\frac{h(\mu_k)}{\lambda(\mu_k)}\geq t-A(k)$$
where $A(k)\rightarrow 0$ as $k\rightarrow\infty$.
\end{lema}

\begin{proof}
By the assumptions in our theorem it is possible to find a finite set $D_k\subset \mathcal{C}_k$ such that
$$\sum_{I(i_1,\ldots,i_k)\in D_k}|I(i_1,\ldots,i_k)|^{t}=Z_k>1.$$
We can now construct a $T^k$ invariant Bernoulli measure,   $\overline{\nu_k}$ by assigning each cylinder in $D_k$ weight $\frac{1}{|D_k|}|I(i_1,\ldots,i_k)|^t$. We can now estimate
\begin{eqnarray*}
\frac{h(\overline{\nu_k},T^k)}{\lambda(\overline{\nu_k},T^k)}&=&\frac{-\sum_{I(i_1,\ldots,i_k)\in D_k} \frac{1}{|D_k|}|I(i_1,\ldots,i_k)|^t\log (\frac{1}{|D_k|}|I(i_1,\ldots,i_k)|^t)}{\lambda(\overline{\nu_k},T^k)}\\
&=&\frac{-\sum_{I(i_1,\ldots,i_k)\in D_k} \frac{1}{|D_k|}|I(i_1,\ldots,i_k)|^t\log |I(i_1,\ldots,i_k)|^t}{\lambda(\overline{\nu_k},T^k)}+\frac{Z_k \log|D_k|}{|D_k|\lambda(\overline{\nu_k},T^k)}\\
&\geq&\frac{t\lambda(\overline{\nu_k},T^k)-C}{\lambda(\overline{\nu_k},T^k)}\geq t-\frac{C}{k\log\xi}.
\end{eqnarray*}
By Lemma \ref{simpest} we have that for $k$ sufficiently large $\frac{\int S_k\phi\text{d}\overline{\nu_k}}{\int S_k\psi\text{d}\overline{\nu_k}}\in (\alpha-2\epsilon,\alpha+2\epsilon)$. To complete the proof we let $\overline{\mu_k}=\frac{1}{k}\sum_{i=0}^{k-1}\overline{\nu_k}\circ T^{-i}$ and note that $\frac{C}{k\log\xi}\rightarrow 0$ as $k\rightarrow\infty$.
\end{proof}
It now follows that for any $\delta>0$ we can find a sequence of $T$-invariant measures $(\overline{\mu_k})_k$ such that $\limsup_{k\rightarrow\infty} \left|t_k-\frac{h(\overline{\mu_k})}{\lambda(\overline{\mu_k})}\right|\leq\delta$ and where
$\frac{\int\phi\text{d}\overline{\mu_k}}{\int\psi\text{d}\overline{\mu_k}}\in (\alpha-2\epsilon,\alpha+2\epsilon)$. Hence, for all $\epsilon>0$ we have
$$\dim J(\alpha)\leq\sup_{\mu\in \tilde{\mathcal{M}}(T)}\left\{\frac{h(\mu)}{\lambda(\mu)} : \frac{\int\phi\text{d}\mu}{\int\psi\text{d}\mu}\in (\alpha-\epsilon,\alpha+\epsilon)\right\}.$$
To complete the proof of the upper bound for $\alpha\in U$ we simply apply Lemma \ref{continuity}.

\section{Proof of the lower bound of Theorem \ref{main}}
Our method to prove the lower bound is to find  an ergodic measure $\mu$ with $\frac{\int\phi\text{d}\mu}{\int\psi\text{d}\mu}=\alpha$, or a sequence of ergodic measures $\mu_n$ with
$\lim_{n\to\infty}\frac{\int\phi\text{d}\mu_n}{\int\psi\text{d}\mu_n}=\alpha$, and then use the fact that the dimension of ergodic measures with finite entropy is $\frac{h(\mu)}{\lambda(\mu)}$, see  \cite[Thereom 4.4.2]{mubook}.
For $\alpha\in U$ we will use the thermodynamic formalism to show that there is an ergodic equilibrium measure with dimension $\delta(\alpha)$, which will show that $\dim J(\alpha)\geq\delta(\alpha)$. For $\alpha\notin U$ we will need to work slightly harder.

We define the function $G_1:\R^3\to\R\cup\{\infty\}$ by
$$G_1(\alpha,q,\delta)=P(q(\phi-\alpha\psi)-\delta\log |T'|).$$
Note that $G_1$ can be infinite and we will adopt the convention $\infty>0$.
We have the following  lemma.

\begin{lema}\label{pospres}
Let $\alpha\in (\alpha_m,\alpha_M)$ and $\delta>0$. If for all $q\in\R$ we have that $G_1(\alpha,q,\delta)> 0$ then $\dim J(\alpha)>\delta$.
\end{lema}

\begin{proof}
Since $\alpha\in (\alpha_m,\alpha_M)$ it follows that
$$\lim_{q\to\infty}G_1(\alpha,q,\delta)=\lim_{q\to-\infty}G_1(\alpha,q,\delta)=\infty.$$
Indeed, this is a consequence of ergodic optimisation results that relate the asymptotic derivative of the pressure to maximising/minimising measures for the potential $\phi-\alpha\psi$ (see for instance \cite[Theorem 1]{JMU}). This means that if we let $q^-:=\inf\{q \in \R :G_1(\alpha,q,\delta)<\infty\}$ and $q^+:=\sup\{q  \in \R :G_1(\alpha,q,\delta)<\infty\}$ then either $\lim_{q\to q^{-}}G(\alpha,q,\delta)=\infty$ or $G(\alpha,q^{-},\delta)<\infty$ and similarly $\lim_{q\to q^{+}}G(\alpha,q,\delta)=\infty$ or $G(\alpha,q^{+},\delta)<\infty$. Thus if $\{q \in \R :G_1(\alpha,q,\delta)<\infty\}\neq\emptyset$ then, by convexity and our assumption,  $G_1(\alpha,q,\delta)$ has a minimum which must be greater than $0$. Therefore, by the approximation property of pressure  (see Proposition \ref{tf_maps}) we can find  $n\in\N$ such that the $T$-invariant compact set $\Lambda_n$ satisfies
$$P_{\Lambda_n}(q(\phi-\alpha\psi)-\delta\log|T'|)>0$$
for all $q\in\R$ and $$\lim_{q\to-\infty}P_{\Lambda_n}(q(\phi-\alpha\psi)-\delta\log|T'|)=\lim_{q\to\infty}P_{\lambda}(q(\phi-\alpha\psi)-\delta\log|T'|)=\infty.$$
Moreover,  the function $q\to P_{\Lambda_n}(q(\phi-\alpha\psi)-\delta\log|T'|)$ is real analytic and strictly convex (see \cite{su, sar}). Thus,  there exists  a unique point $q_c \in \R$ such that
$$\frac{\partial }{\partial q}P_{\Lambda_n}(q,\alpha,\delta) \Big|_{q=q_c}=0.$$
Denote by  $\mu_c$  the unique equilibrium state for the potential $q_c(\phi-\alpha\psi)-\delta\log|T'|$ restricted to $\Lambda_n$. Note that such a measure exists because the space $\Lambda_n$ is compact  and the potential H\"older. Then $\frac{\int\phi \text{d}\mu_c}{\int\psi\text{d}\mu_c}=\alpha$ and so
$$h(\mu_c)-\delta\lambda(\mu_c)>0.$$
 Since $\mu_c$ is ergodic we have that $\mu_c(J(\alpha))=1$ and $\dim\mu_c=\frac{h(\mu_c)}{\lambda(\mu_c)}\geq\delta$. This completes the proof.
\end{proof}
The following lemma now completes the proof of the lower bound for the case when $\alpha\in U$.

\begin{lema} \label{lem:4.2}
If $\alpha\in U$   then for all $q\in\R$ and any $\epsilon>0$ we have that
$$P(q(\psi-\alpha\phi)-(\delta(\alpha)-\epsilon)\log |T'|)>0.$$
\end{lema}

\begin{proof}
By the definition of $\delta(\alpha)$ there exists a $T$-invariant measure $\mu$ such that $\frac{\int\phi\text{d}\mu}{\int\psi\text{d}\mu}=\alpha$ and $\frac{h(\mu)}{\lambda(\mu)}>\delta(\alpha)-\epsilon$. Thus, for any $q\in\R$, the variational principle yields
\begin{eqnarray*}
P(q(\phi-\alpha\psi)-(\delta(\alpha)-\epsilon)\log |T'|)&\geq&\\ q \left(\int\phi\text{d}\mu-\alpha\int\psi\text{d}\mu \right)-(\delta(\alpha)-\epsilon)\lambda(\mu)+h(\mu)
&=&\\-(\delta(\alpha)-\epsilon)\lambda(\mu)+h(\mu)>0.
\end{eqnarray*}
\end{proof}

For $\alpha\notin U$ we cannot use the argument in Lemma \ref{lem:4.2}. Instead, we need to use a sequence of measures. Fix $\alpha\notin U$ and let
$$s=\lim_{\epsilon\to 0}\sup_{\mu \in \tilde{\mathcal{M}}_T}\left\{\frac{h(\mu)}{\lambda(\mu)}:\left|\frac{\int\phi\text{d}\mu}{\int\psi\text{d}\mu}-\alpha\right|\leq \epsilon\right\}.$$
\begin{lema}\label{seqermeas}
There exists a sequence of $T$-ergodic measures $(\mu_n)_n$ such that,
$\lim_{n\to\infty}\frac{h(\mu_n)}{\lambda(\mu_n)}=s$ and $\lim_{n\to\infty}\frac{\int \phi\text{d}\mu_n}{\int\psi\text{d}\mu_n}=\alpha$.
\end{lema}

\begin{proof}
It follows from the definition of $s$ that we can find a sequence of invariant measures with this property. In order to construct a sequence of ergodic measures with the desired property we proceed as follows.  Let $\mu$ be an invariant measure with $\max\{h(\mu),\int \psi\text{d}\mu \}<\infty$. For any $n\in\N$ we can define a $T^n$-ergodic Bernoulli measure $\eta_n$ with $\eta_n([i_1,\ldots,i_n])=\mu([i_1,\ldots,i_n])$ for each $n\in\N$. If we let $\nu_n=\frac{1}{n}\sum_{i=0}^{n-1}\eta_n\circ T^{-1}$ then $\nu_n$ is a $T$-ergodic measure. Moreover $\lim_{n\to\infty}\frac{h(\nu_n)}{\lambda(\nu_n)}=\frac{h(\mu)}{\lambda(\mu)}$ and $\lim_{n\to\infty}\frac{\int \phi\text{d}\nu_n}{\int\psi\text{d}\nu_n}=\frac{\int \phi\text{d}\mu}{\int\psi\text{d}\mu}$.
\end{proof}

Thus we have a sequence of $T$-invariant ergodic measures $(\mu_n)_n$
 such that
 \[ \lim_{n\to\infty}\frac{h(\mu_n)}{\lambda(\mu_n)}=s \text{ and } \lim_{n\to\infty}\frac{\int\phi\text{d}\mu_n}{\int\psi\text{d}\mu_n}=\alpha.\]
 If $\limsup_{n\to\infty}\int\psi\text{d}\mu_n<\infty $ then we can adapt the method in Lemma \ref{continuity} to find a sequence of $T$-invariant measures $(\nu_n)_n$ with $\lim_{n\to\infty}\frac{h(\nu_n)}{\lambda(\nu_n)}=s$ and $\frac{\int\phi\text{d}\nu_n}{\int\psi\text{d}\nu_n}=\alpha$. We can then proceed as in the case when $\alpha\in U$. On the other hand, if $\limsup_{n\to\infty}\int\psi\text{d}\mu_n=\infty$ we can adapt the technique in  \cite[Section 7]{ij} which is in turn based on the method in Gelfert and Rams, \cite{GR} to prove the following proposition to complete the proof.

\begin{prop}\label{wmeasure}
If there exists a sequence of $T$-ergodic measures $(\mu_n)_n$ such that $\lim_{n\to\infty}\frac{h(\mu_n)}{\lambda(\mu_n)}=s$, $\frac{\int\phi\text{d}\mu_n}{\int\psi\text{d}\mu_n}=\alpha$ and $\lim_{n\to\infty}\int\psi\text{d}\mu_n=\infty$ then there exists a measure $\nu$, such that $\dim\nu=s$ and $\lim_{n\to\infty}\frac{S_n\phi(x)}{S_n\psi(x)}=\alpha$ for $\nu$-almost all $x$.  
\end{prop}
Note that the measure $\nu$ is not required to be invariant and in certain cases an invariant measure with the required property will not exist. 

\subsection*{Proof of Proposition \ref{wmeasure}}
Note that it suffices to proof this in the case where $\alpha_m>0$ since if this is not the case we can add a suitable constant multiple of $\psi$ to $\phi$ to ensure it is the case. We let
$$\alpha_n:=\frac{\int\phi\text{d}\mu_n}{\int\psi\text{d}\mu_n}\text{ and }s_n=\frac{h(\mu_n)}{\lambda(\mu_n)}.$$  
We may also assume that the sequence $s_n$ is monotone increasing. Finally we fix a sequences of positive real numbers $0<\delta_n$ such that $\prod_{n=1}^{\infty}(1-\delta_n)>0$ and fix $0<\epsilon<\inf_{n\in\N}\min\{\alpha_n,s_n,\int \psi_n\text{d}\mu_n\}$.
\begin{lema}
For each measure $\mu_n$ there exists a set $J_n$ and $j_n\in\N$ such that $\mu_n(J_n)>1-\delta_n$ and for all $x=\Pi(\underline{i})\in J_n$ and $j\geq j_n$ we have that
\begin{enumerate}
\item
For all $y\in I(i_1,\ldots i_j)$, $S_j\psi(y)\in \left(j\left(\int \psi\text{d}\mu_n-\epsilon/2^n\right)\right),j\left(\int\psi\text{d}\mu_n+\epsilon/2^n\right)$ and $S_j\phi(y)/S_j\psi(y)\in \left(\alpha_n-\epsilon/2^n,\alpha_n+\epsilon/2^n\right)$.
\item
$|I(i_1,\ldots,i_j)|\in (j(\lambda(\mu_n)-\epsilon/2^n),j(\lambda(\mu_n)+\epsilon/2^n))$ and $\frac{\log(\mu_n(I(i_1,\ldots,i_j)))}{\log|I(i_1,\ldots,i_j)|}\in ((s_n-\epsilon/2^n,s_n+\epsilon/2^n)$.
\end{enumerate}
\end{lema}
\begin{proof}
This is a straightforward consequence of the Birkhoff ergodic theorem and the Shannon-McMillan-Brieman Theorem combined with Egorov's Theorem and our assumptions on $\phi,\psi,\log|T'|$.
\end{proof}
We now define a sequence of natural numbers $k_n$ using the following inductive procedure. We let $k_1$ satisfy the following conditions: 
\begin{enumerate}
\item
$k_1\geq j_1$,
\item
$\epsilon_1k_1\int \psi\text{d}\mu_1\geq 4j_2\alpha_2\int\psi\text{d}\mu_2$, 
\item
$\epsilon_1k_1\lambda(\mu_1)\geq 4j_2\alpha_2\lambda(\mu_2)$,
\item
$k_1\epsilon\geq 4j_2\lambda(\mu_2)$.
\end{enumerate}
For $n\geq 1$ we can choose $k_{n+1}$ to satisfy:
\begin{enumerate}
\item
$(k_{n+1}-k_n)\epsilon\int\psi\text{d}\mu_{n+1}\geq\left|2^{n+2}k_n(\alpha_{n+1}-\alpha_n)\int\psi\text{d}\mu_{n}\right|$,
\item
$\epsilon k_{n+1}\int \psi\text{d}\mu_{n+1}\geq 2^{n+2}j_{n+2}\alpha_{n+2}\int\psi\text{d}\mu_{n+2}$,
\item
$(k_{n+1}-k_n)\epsilon\lambda(\mu_{n+1})\geq |2^{n+2}k_n(s_{n+1}-s_n)\lambda(\mu_n)|$,
\item
$\epsilon k_{n+1}\lambda(\mu_{n+1})\geq 2^{n+2}j_{n+2}s_{n+2}\lambda(\mu_{n+2})$,
\item
$(k_{n+1}-k_n)\epsilon\geq 2^{n+2}|\lambda(\mu_{n+1})-\lambda(\mu_n)|$
\item
$k_{n+1}\epsilon\geq 2^{n+2}j_{n+2}\lambda(\mu_{n+2})$.
\end{enumerate}
Now consider a point $x\in J_1\cap \sigma^{-k_1}(J_2)\cap \sigma^{-k_2}(J_3)\cap\cdots$. By the construction of the sets $J_n$ and the values $k_n$ it follows that $\lim_{n\to\infty}\frac{S_n\phi(x)}{S_n\psi(x)}=\alpha$. We can also define a measure supported on this set as follows. 
Let $\eta_n$ denote the measure such that 
$$\eta_n(I(i_1,\ldots,i_{k_n}))=\left\{\begin{array}{ccc}0&\text{ if }&I(i_1,\ldots,i_{k_n})\cap J_n=\emptyset\\
\mu_n(I(i_1,\ldots,i_{k_n}))&\text{ if }&I(i_1,\ldots,i_{k_n})\cap J_n\neq\emptyset.\end{array}\right.$$
defined on the algebra consisting of $k_n$ level cylinders
It follows that $\eta_n(J_n)\geq 1-\delta_n$ and thus if we define the measure (with respect to the Borel sigma algebra)
$$\eta=\eta_1\otimes\eta_2\circ T^{-k_1}\otimes \eta_3\circ T^{-k_2}\otimes\cdots$$
we will have 
$$\eta(J_1\cap \sigma^{-k_1}(J_2)\cap \sigma^{-k_2}(J_3)\cap\cdots)\geq\prod_{n=1}^{\infty}(1-\delta_n)>0$$
and we can normalise to a measure $\nu$. Let $C=\left(\prod_{n=1}^{\infty}(1-\delta_n)\right)^{-1}$.
\begin{lema}
For $\nu$ almost all $x$ we have that
$$\lim_{n\to\infty}\frac{S_n\phi(x)}{S_n\psi(x)}=\alpha\text{ and }\liminf_{r\to 0}\frac{\log\nu(B(x,r))}{\log r}\geq s.$$
\end{lema}

\begin{proof}
By the construction of the measure $\nu$ we have that for $\nu$ almost all $x=\pi(\underline{i})$, where $\underline{i} \in \Sigma$, is  that
$$\lim_{n\to\infty}\frac{S_n\phi(x)}{S_n\psi(x)}=\alpha\text{ and }\lim_{n\to\infty}\frac{\log \nu(I(i_1,\ldots,i_n))}{\log |I(i_1,\ldots,i_n)|}=s.$$
To complete the proof we will fix $r>0$ sufficiently small. We fix $n\geq 2$ and first of all consider the case when 
$$e^{-k_n(\lambda(\mu_n)-\epsilon/2^{n-1})}\geq r\geq e^{-k_n(\lambda(\mu_n)-\epsilon/2^{n-1})-j_{n+1}(\lambda(\mu_{n+1})+\epsilon/2^n)}.$$
In this case $B(x,r)$ contains at most  $e^{j_{n+1}(\lambda(\mu_{n+1})+\epsilon/2^n)}$ sets of the form $[i_1,\ldots, i_{k_n}]$. Thus
\begin{eqnarray*}
\log\mu(B(x,r)) &\leq& \log C+k_n\lambda(\mu_n)(s_n-\epsilon/2^{n-1})+j_{n+1}(\lambda(\mu_{n+1})+\epsilon/2^n) \\ &\leq& -k_n\lambda(\mu_n)(s_n-\epsilon_n)\leq \log C+(s_n-\epsilon_n)\log r.
\end{eqnarray*}
Now consider the case where $j_{n+1}\leq k\leq k_{n+1}-k_n$ and 
$$e^{-k_n(\lambda(\mu_n)-k(\lambda(\mu_{n+1})-\epsilon/2^n))}\geq r>e^{-k_n(\lambda(\mu_n)-(k+1)(\lambda(\mu_{n+1}-\epsilon/2^{n+1}))}.$$
Thus we have that
\begin{eqnarray*}
\log \mu(B(x,r)) &\leq& \log C+\log\lambda(\mu_{n+1})+\log\mu(I(i_1,\ldots,i_{k_n+k}))\\
&\leq& \log C\log\lambda(\mu_{n+1})+k_n\lambda(\mu_n)(s_n-\epsilon/2^{n-1})+k\lambda(\mu_{n+1})(s_{n+1}-\epsilon/2^{n-1})\\
&\leq& \log C+(s_n-\epsilon/2^{n-1})\log r.
\end{eqnarray*}
Finally note that by the definition of $k_n$ 
$$e^{-k_n(\lambda(\mu_n)-k_{n+1})-k_n(\lambda(\mu_{n+1})-\epsilon_{n+1}/2)}\leq e^{-k_n(\lambda(\mu_n)-\epsilon)}n
$$ 
and so we have considered all possibles case for $r<e^{-k_2(\lambda(\mu_1)-\epsilon_1)}$ and the result follows.
\end{proof} 
 
The proof of Proposition \ref{wmeasure} is now complete. 
\section{Proof of Theorem \ref{analytic}}
We now turn to the question of when the spectrum is analytic. Throughout this section $\phi,\psi$ are both locally H\"{o}lder,  $\alpha_M<\infty$ , $\phi/\psi$ is uniformly bounded above
 and
$$\lim_{x\to 0}\frac{\psi(x)}{\log |T'(x)|}=\infty.$$
An important consequence of our assumptions is that if $\mu$ is a $T$-invariant measure for which  $\frac{\int\phi\text{d}\mu}{\int\psi\text{d}\mu}\in U$ then by Lemma \ref{bounded} $\int\psi\text{d}\mu$ cannot be too large and thus $\lambda(\mu)$ cannot be too large.

\begin{lema}\label{bounded2}
For any $C_1>0$ there exists $C_2>0$ such that if $\mu$ is a $T-$invariant measure for which $\lambda(\mu)\geq C_2$ then $\int\psi\text{d}\mu\geq C_1\lambda(\mu)$.
\end{lema}

\begin{proof}
Let $\delta >0$, then there exists $A \in (0,1)$ such that if $ x \in (0, A)$ then
\begin{equation*}
\frac{\log |T'(x)|}{\psi(x)}<\delta.
\end{equation*}
Moreover, there exists a constant $C>0$ such that if $x \in (A, 1)$ then $\log |T'(x)| <C$. Let $\mu$ be a $T-$invariant measure satisfying $ \lambda(\mu) \geq C$, then we have that
\begin{eqnarray*}
\lambda(\mu)= \int_0^1 \log |T'(x)| d \mu = \int_0^A \log |T'(x)| d \mu  + \int_A^1 \log |T'(x)| d \mu &\\ \leq  \int_0^A \log |T'(x)| d \mu + \mu([A, 1]) C  \leq \int_0^A \log |T'(x)| d \mu +C.
\end{eqnarray*}
That is
\begin{equation*}
\lambda(\mu)-C \leq \int_0^A \log |T'(x)| d \mu.
\end{equation*}
We thus have,
\begin{eqnarray*}
\delta  \geq \frac{\int_0^A \log |T'(x)| d \mu }{\int_0^A \psi d \mu } \geq \frac{\lambda(\mu)-C}{\int_0^A \psi d \mu}.
\end{eqnarray*}
Therefore, since $\psi >0$, we have
\begin{equation} \label{psi}
\int_0^1 \psi d \mu \geq \int_0^A \psi d \mu \geq \frac{\lambda(\mu)-C}{\delta}.
\end{equation}
Thus, given $C_1 >0$ choose $\delta >0$ such that $1-\delta C_1 >0$ and let
\begin{equation*}
C_2 > \frac{C}{1-\delta C_1}.
\end{equation*}
Therefore, if  $\lambda(\mu) > C_2$ then $\lambda(\mu) (1-\delta C_1) > C$, which implies that
\begin{equation*}
\frac{\lambda(\mu)-C}{\delta} > C_1 \lambda(\mu).
\end{equation*}
Combining this with equation \eqref{psi} we obtain that
\begin{equation*}
\int_0^1 \psi d \mu \geq \int_0^A \psi d \mu \geq \frac{\lambda(\mu)-C}{\delta} \geq C_1 \lambda(\mu),
\end{equation*}
which completes the proof.
\end{proof}

\begin{lema}\label{finiteness}
For $\alpha\in U$ we have that for any $\delta>0$ either
\begin{enumerate}
\item $G_1(\alpha,q,\delta)<\infty $ for all $q>0$ or
\item $G_1(\alpha,q,\delta)<\infty$ for all $q<0$.
\end{enumerate}
\end{lema}

\begin{proof}
Since $\alpha\in U$ we know that either $\alpha>\overline{\alpha}$ or $\alpha<\underline{\alpha}$. To start we will assume that $\alpha>\overline{\alpha}$. We fix $\gamma\in (\overline{\alpha},\alpha)$ and $q>0$. By the variational principle we need to show that there is a uniform upper bound on
$$h(\mu)+q\left(\int \phi\text{d}\mu-\alpha \int\psi\text{d}\mu\right)-\delta \lambda(\mu)$$
for all $T$-invariant probability measures $\mu$. We will split the set of invariant measures into two sets depending on whether or not $$\frac{\int\phi\text{d}\mu}{\int\psi\text{d}\mu}\leq\gamma.$$

Firstly, if $\int\phi\text{d}\mu \leq \gamma \int\psi\text{d}\mu$ then $q(\int \phi\text{d}\mu-\alpha\int\psi\text{d}\mu)\leq q(\gamma-\alpha)\int\psi\text{d}\mu<0$. Furthermore, by taking $C_1=(-q(\gamma-\alpha))^{-1}$ in Lemma \ref{bounded2} it follows that there exists $C_2>0$ such that if $\lambda(\mu)\geq C_2$ we have that (using Ruelle's inequality)
\begin{equation*}
\int\psi\text{d}\mu \geq C_1 \lambda(\mu) \geq C_1 h(\mu) = \frac{1}{-q(\gamma-\alpha)} h(\mu).
\end{equation*}
That is,
\[ -q(\gamma-\alpha) \int\psi\text{d}\mu \geq  h(\mu).\]
We therefore have
\begin{eqnarray*}
h(\mu)+q\left(\int \phi\text{d}\mu-\alpha \int\psi\text{d}\mu\right)-\delta \lambda(\mu)  \leq & \\
 -q(\gamma-\alpha) \int\psi\text{d}\mu + q(\gamma-\alpha)\int\psi\text{d}\mu-\delta \lambda(\mu) = -\delta \lambda(\mu) \leq 0.
\end{eqnarray*}
On the other hand, if
$$\frac{\int\phi\text{d}\mu}{\int\psi\text{d}\mu}\geq\gamma.$$
By Lemma \ref{bounded} there exists a constant $K'$ such that $\int\psi\text{d}\mu <K'$. We also have that there exists a uniform upper bound for $K''>0$ for
\[ \frac{\lambda(\mu)}{\int\psi\text{d}\mu} < K''. \]
Therefore, there exists $K>0$ such that $\max\left\{h(\mu),\lambda(\mu),\int\psi\text{d}\mu\right\}\leq K$. This means that
$$h(\mu)+q\left(\int \phi\text{d}\mu-\alpha \int \psi\text{d}\mu\right)-\delta(\alpha)\lambda(\mu)\leq K+q\alpha_M K.$$
So for $\alpha\in (\overline{\alpha},\alpha_M)$ we have that for all $q>0$  the function $G_1(\alpha,q,\delta)$ is bounded above,
\[G_1(\alpha,q,\delta) < \infty.  \]
For the case where $\alpha\in (\alpha_m,\underline{\alpha})$ we fix $q<0$ and $\gamma\in (\alpha,\underline{\alpha})$. The same argument allow us to conclude that $G_1(\alpha,q,\delta)<\infty$.

%
%

\end{proof}

As a result of Lemma \ref{finiteness} we can investigate the behaviour of the function $G_1$ when $\delta=\delta(\alpha)$ and $\alpha\in U$.

\begin{lema}\label{options}
For $\alpha\in U$ we have that $G_1(\alpha,q,\delta(\alpha))\geq 0$ for all $q \in \R$ and that either
\begin{enumerate}
\item\label{case1}
 there exists a unique $q_c\neq 0$ such that $G_1(\alpha,q,\delta(\alpha))=0$ and $$\frac{\partial }{\partial q}G_1(q,\delta(\alpha),\alpha) \big|_{q=q_c}=0$$ or
\item
 $\delta(\alpha)=\dim\Lambda$ and $P(-\delta(\alpha)\log|T'|)= 0$.
\end{enumerate}
\end{lema}

\begin{proof}
We first prove that  for all $q \in \R$ we have $G_1(\alpha,q,\delta(\alpha))\geq 0$. Note that since $\alpha\in U$, there exists a constant $C>0$ and a sequence of $T$-invariant measures $(\mu_n)_n$ such that $\lambda(\mu_n)\leq C$, $\frac{\int \phi\text{d}\mu_n}{\int\psi\text{d}\mu_n}=\alpha$ and $\lim_{n\rightarrow\infty}\frac{h(\mu_n)}{\lambda(\mu_n)}=\delta(\alpha)$. We thus have
for all $q \in \R$ and $n \in \N$ that
$$G_1(\alpha,q,\delta(\alpha)) \geq h(\mu_n)-\delta(\alpha)\lambda(\mu_n).$$
Letting $n$ tend to infinity we obtain
 $$G_1(\alpha,q,\delta(\alpha))\geq 0.$$
If $\delta(\alpha)=\dim\Lambda$ the above argument together with  Bowen's formula,
$$P(-\dim\Lambda\log|T'|)\leq 0,$$
  implies that $P(-\delta(\alpha)\log|T'|)= 0$.
%
%

 From now on we can suppose that $\delta(\alpha)<\dim\Lambda$. In particular,
 $$0 <P(-\delta(\alpha)\log|T'|),$$
note that it could be infinite. Moreover,  since $\alpha_m<\alpha<\alpha_M$ we have that $$\lim_{q\to\infty}G_1(\alpha,q,\delta(\alpha))=\lim_{q\to -\infty}G_1(\alpha,q,\delta(\alpha))=\infty.$$
Now, in order to obtain a contradiction, suppose there exists $C_1>0$ such that
 $$G_1(\alpha,q,\delta(\alpha))\geq C_1>0$$
 for all $q\in\R$.
 Then we can find a compact $T$-invariant subset $\Lambda_n$, where  $\Lambda_n \subset \Lambda$, such that for every $q \in \R$ we have
 $$ 0 < C_1/2 \leq P_{T|\Lambda{_n}}(q(\phi-\alpha\psi)-\delta(\alpha)\log|T'|)< \infty$$
 and such that the following holds
 \begin{eqnarray*}
 \lim_{q\to-\infty}P_{T|\Lambda_{n}}(q(\phi-\alpha\psi)-\delta(\alpha)\log|T'|)=\infty \text{ and } &\\
 \lim_{q\to\infty}P_{T|\Lambda_{n}}(q(\phi-\alpha\psi)-\delta(\alpha)\log|T'|)=\infty.
 \end{eqnarray*}
 Thus there will exist a turning point $q_t \in \R$ and the equilibrium state $\mu_t$ associated to $q_t(\phi-\alpha\psi)-\delta(\alpha)\log|T'|$ restricted to $\Lambda_n$ satisfying
 $\frac{\int\phi\text{d}\mu_t}{\int\psi\text{d}\mu_t}=\alpha$ and $\frac{h(\mu_t)}{\lambda(\mu_t)}>\delta(\alpha)$, which is a contradiction.

Thus, there exists a nonzero point $q_c \in \R$ such that
 $G_1(\alpha,q_c,\delta(\alpha))=0.$ Since $G_1(\alpha,0,\delta(P(-\delta(\alpha)\log|T'|)>0$ and
 $$\lim_{q\to\infty}G_1(\alpha,q,\delta(\alpha))=\lim_{q\to -\infty}G_1(\alpha,q,\delta(\alpha))=\infty
 $$
 it follows from Lemma \ref{finiteness} that $q_c$ must be a turning point and the result follows.

\end{proof}

We can now split the region $(\alpha_m,\alpha_M)$ up into three intervals depending on whether $\alpha$ in $U$ or not and if so which case of Lemma \ref{options} is true.

\begin{lema}\label{partition}
We can write $(\alpha_m,\alpha_M)=J_1\cup J_2\cup J_3$ where $J_1,J_3$ are intervals or the empty set, $J_2$ is an interval or a single point such that $E=[\underline{\alpha},\overline{\alpha}]\subset J_2$ and
\begin{enumerate}
\item
For $\alpha\in J_2$, $\dim J(\alpha)=\dim\Lambda$.
\item
For $\alpha\in J_1\cup J_3$ we have that there exists a unique $q_c\neq 0$ such that $P(q_c(\phi-\alpha\psi)-\delta(\alpha)\log|T'|=0$ and $\frac{\partial}{\partial q}G_1(q_c,\delta(\alpha),\alpha)=0$
\end{enumerate}
\end{lema}

\begin{proof}
We define
$$J_2=\overline{\{\alpha\in (\alpha_m,\alpha_M):P(q(\phi-\alpha\psi)-(\dim\Lambda-\epsilon)\log |T'|)\geq 0   \text{ for } q \in \R \text{ and }\epsilon>0\}}$$
and also let
\begin{eqnarray*}
J_1=\{\alpha\in (\alpha_m,\alpha_M):\alpha<\gamma \text{ for every }  \gamma\in J_2\}\text{ and } &\\
J_3=\{\alpha\in (\alpha_m,\alpha_M):\alpha>\gamma \text{ for every } \gamma\in J_2\}.
\end{eqnarray*}
For $\alpha\in J_2$ we have combining  Lemma \ref{pospres} with Theorem \ref{main} that $\dim J(\alpha)=\dim\Lambda$.

If $\alpha_1,\alpha_2\in J_2$ and $\gamma\in (\alpha_1\,\alpha_2)$ then $\gamma\in J_2$ by the convexity of the pressure function. Therefore $J_2$ is either a single point or an interval and it thus follows that both $J_1$ and $J_3$ are either empty or intervals.

To show that $E=[\underline{\alpha},\overline{\alpha}]\subset J_2$   we fix $\alpha\in E$ and choose $\epsilon>0$ and  $\gamma$  satisfying that there exist a $T$-invariant measure $\mu$ such that $\lambda(\mu)<\infty$, $\frac{h(\mu)}{\lambda(\mu)}>\dim\Lambda-\epsilon/2$ and $\frac{\int\phi\text{d}\mu}{\int\psi\text{d}\mu}=\gamma$. Suppose that $\gamma>\underline{\alpha}$ and let $\alpha\in (\underline{\alpha},\gamma)$. Since $\alpha\in E$ we can find a sequence of $T$-invariant measures $\nu_n$ and $0<p_n<1$ such that
\[\lim_{n\to\infty}\frac{\int\phi\text{d}\nu_n}{\int\psi\text{d}\nu_n}=\underline{\alpha} , \lim_{n\to\infty}p_n\lambda(\nu_n)=0 \text{ and } \lim_{n\to\infty}p_n\int\psi\text{d}\nu_n=\infty.\]
 Thus, if we consider the measures $\eta_n=p_n\nu_n+(1-p_n)+\nu$ then
$\lim_{n\to\infty}\frac{\int\phi\text{d}\eta_n}{\int\psi\text{d}\eta_n}=\underline{\alpha}$ and $\limsup_{n\to\infty}\frac{h(\eta_n)}{\lambda(\eta_n)}\geq\dim\Lambda-\epsilon/2$. Therefore, by taking an appropriate convex combination,  we can find a $T$-invariant measure $\nu$ such that
\[\frac{h(\nu)}{\lambda(\nu)}>\dim\Lambda-\epsilon \text{ and } \frac{\int\phi\text{d}\nu}{\int\psi\text{d}\nu}=\alpha.\]
By the variational principle, this means that  for all $q\in\R$  we have
\[P(q(\phi-\alpha)-(\dim\Lambda-\epsilon)\log |T'|)>0.\]
 The case when $\alpha>\gamma$ can be dealt with analogously by using the interval $(\gamma,\overline{\alpha})$. Finally note that if $\gamma=\underline{\alpha}=\overline{\alpha}$ then we can verify that $\gamma\in J_2$ directly by the variational principle.

Now fix $\alpha\in J_1$. This means that $\alpha\in U$ and that there exists $q\in\R$ such that
$$P(q(\phi-\alpha\psi)-(\dim\Lambda)\log |T'|)<0.$$
Thus, by Lemma \ref{options} we know that $\delta(\alpha)<\dim\Lambda$ and since then $$P(-\delta(\alpha)\log |T'(x)|)>0$$
 we must be in case \ref{case1} of Lemma \ref{options}. The case when $\alpha\in J_3$ can be dealt with analogously. This completes the proof.
\end{proof}

The three intervals $J_1,J_2,J_3$ will be exactly the three intervals in the statement of Theorem \ref{analytic}. To complete the proof of Theorem \ref{analytic} we simply need to prove that $\alpha\to\dim J(\alpha)$ varies analytically on $J_1$ and $J_3$.

\begin{lema}\label{analicity}
The function $\alpha\to J(\alpha)$ is analytic in $J_1$ and $J_3$.
\end{lema}

\begin{proof}
We will proof this result for the interval $J_1$, since the proof for the interval $J_3$ is analogous. Note that for $\alpha\in J_1$ we have that $\dim J(\alpha)=\delta(\alpha)$. We will let $G_2:J_2\times\R^{+}\times (0,\dim\Lambda)\to\R$ be defined by
$$G_2(\alpha,q,\delta)=\frac{\partial}{\partial q}G_1(\alpha,q,\delta)$$
and note that since $G_1$ is finite throughout the specified region $G_2$ is well defined.
Let $G:J_2\times\R^{+}\times (0,\dim\Lambda)\to \R^2$ be defined by
$$G(\alpha,q,\delta)=(G_1(\alpha,q,\delta),G_2(\alpha,q,\delta)).$$
For each $\alpha\in J_1$ by Lemma \ref{partition} that there exists a unique $q(\alpha) \in \R$ such that
\[ G(\alpha,q(\alpha),\delta(\alpha))=(0,0) \text{ and } \dim J(\alpha)=\delta(\alpha).\]
Note that $G$ is finite and varies analytically in each of the three variables $q,\alpha,\delta$ throughout its range. Thus, to complete the proof we wish to apply the Implicit Function Theorem. To be able to do this it suffices to show that the matrix

$$\begin{pmatrix}\frac{\partial G_1}{\partial \delta}&\frac{\partial G_1}{\partial q}\\ \frac{\partial G_2}{\partial \delta}&\frac{ \partial G_2}{\partial q}\end{pmatrix}$$
is invertible at each point $(\alpha,q(\alpha),\delta(\alpha))$. At such points $\frac{\partial{G_1}}{\partial q}=0$ and $\frac{\partial G_1}{\partial \delta}<0$ (Indeed, it corresponds to the Lyapunov exponent with a minus sign). So we need to show that $\frac{\partial{G_2}}{\partial q}$ is nonzero at $(\alpha,q(\alpha),\delta(\alpha))$. If $\phi-\alpha\psi$ is not cohomologous to a constant then the function $G_1$ is strictly convex in the variable $q$ and the proof is complete. To deduce that $\phi-\alpha\psi$ is not cohomologous to a constant note that $\alpha_m<\alpha<\alpha_M$ and thus there exist $T$-invariant measures $\mu_1$ and $\mu_2$ such that $\int(\phi-\alpha\psi)\text{d}\mu_1<0$ and $\int(\phi-\alpha\psi)\text{d}\mu_2>0$ therefore $\phi-\alpha\psi$ cannot be cohomologous to a constant. 
\end{proof}

\section{Discontinuities in the spectrum and applications to non-uniformly hyperbolic systems}

In this section we show that in the setting of Theorem \ref{analytic} it is possible that the function $\alpha\to\dim J(\alpha)$ is discontinuous at a point in $(\alpha_m,\alpha_M)$.
We stress that this is a new phenomenon that does not occur in the uniformly hyperbolic setting with regular potentials. Here we not only establish conditions for this to happen, but also exhibit very natural non-uniformly hyperbolic dynamical systems where these conditions are satisfied and therefore, regular potentials have discontinuous spectrum.

We will denote by $\mu_{SRB}$ the T-invariant measure of maximal dimension and by
 $s_{\infty}:=\inf\{s : P(-s\log|T'|) < \infty \}$. We have the following result,
\begin{prop}\label{discont}
Let $\phi\in\mathcal{R}$ and $\psi\in\mathcal{R}_{\eta}$ be such that
\[\lim_{x\to 0}\frac{\psi(x)}{\log |T'(x)|}=\infty\]
Assume that
\begin{enumerate}
\item $\lim_{x\to 0}\frac{\phi(x)}{\psi(x)}=0$¬
\item there exists a $T$-invariant probability measure, $\mu$, such that $\int\phi\text{d}\mu<0$;
\item $\dim\Lambda>s_{\infty}$,
\item  $\int \phi\text{d}\mu_{\text{SRB}}>0$
\end{enumerate}
then the function $\alpha\to \dim J(\alpha)$ is discontinuous at $\alpha=0$.
\end{prop}

\begin{proof}
To start we let $\alpha_*=\frac{\int\phi\text{d}\mu_{\text{SRB}}}{\int\psi\text{d}\mu_{\text{SRB}}}$ and note that we have under the notation of Theorem \ref{analytic} that $J_1=[\alpha_m,0]$, $J_2=[0,\alpha_*]$ and $J_3=[\alpha_*,\alpha_M]$. It also follows from Theorem \ref{analytic} that $\dim J(0)=\dim\Lambda$. Assume by way of contradiction that the function  $\alpha \mapsto \dim (J(\alpha))$ is continuous at $\alpha=0$. In particular it is continuous from the left. That is 
$$\lim_{\alpha \to 0^{-} }  \dim J(\alpha)= \dim \Lambda$$
and thus there exists a sequence $(\alpha_n)_n$ such that for every $n \in N$ we have $\alpha_n  \in (\alpha_m, 0)$ and
$$\lim_{n \to \infty} \dim J(\alpha_n) =\dim  \Lambda.$$
In virtue of the variational principle there exists a sequence of $T-$invariant measures $(\mu_n)_n$ such that 
\begin{eqnarray*}
\left( \frac{\int \phi \text{d} \mu_n}{\int \psi \text{d} \mu_n} -  \alpha_n \right) \leq \frac{1}{n}  \textrm{ and } \left( \frac{h(\mu_n)}{\lambda(\mu_n)} -  \dim J(\alpha_n) \right) \leq \frac{1}{n}
\end{eqnarray*}
In particular, for each $n \in \N$ we have $\int\phi\text{d}\mu_n<0$, $\lim_{n\to\infty}\frac{\int \phi\text{d}\mu_n}{\int\psi\text{d}\mu_n}=0$ and $\lim_{n\to\infty}\frac{h(\mu_n)}{\lambda(\mu_n)}=\dim\Lambda$.
By \cite[Proposition 6.1]{fjlr} such a sequence will have a weak * limit $\nu$ such that $\frac{h(\nu)}{\lambda(\nu)}=\dim\Lambda$. Moreover via the semi-continuity of the map $\mu\to\int\phi\text{d}\mu$ (\cite[Lemma 1]{JMU}) we have that $\int\phi\text{d}\nu\leq 0$. Thus $\nu$ is an equilibrium state for the potential $-(\dim\Lambda)\log|T'|$ and since this equilibrium state is unique we must have $\nu=\mu_{\text{SRB}}$. However, $\int\phi\text{d}\mu_{\text{SRB}}>0$ and $\int\phi\text{d}\nu\leq 0$ which is obviously a contradiction. Therefore there exists no such sequence of measures and $\alpha\to \dim J(\alpha)$ is discontinuous at $\alpha=0$.
%
\end{proof}

\subsection{Manneville-Pomeau}
In this subsection we discuss an example that illustrates how our results Theorem \ref{analytic} and Proposition \ref{discont} can be applied to certain classes of non-uniformly expanding maps.
In \cite{jjop}  the authors proved a variational principle for Birkhoff averages for continuous potentials and certain non-uniformly expanding maps. A particular case of this is the Manneville-Pomeau map, $F:[0,1] \mapsto [0,1]$,  which is the map defined by $F(x)= x +x ^{1+ \beta} \mod 1$, where $0 < \beta < 1$. This map has an indifferent fixed point at $x=0$ and it has an absolutely continuous (with respect to Lebesgue)  invariant probability measure. We will denote $A=\cup_{n=0}^{\infty}F^{-n}(\{0\})$ and let $0<t<1$ satisfy $t+t^{1+\beta}=1$.  We define a partition $\mathcal{P}_1=\{[0,t],[t,1]\}$ and $\mathcal{P}_n=\bigvee_{0}^{n-1}T^{-n}\mathcal{P}_1$. For a function $f:[0,1]\to\R$ let
$$\text{var}_n(f)=\sup_{P\in\mathcal{P}_n}\sup\{|f(x)-f(y)|:x,y\in P\}.$$
We will assume that there exists $A>0$ and $0<\theta<1$ such that $\text{var}_n(f)\leq A\theta^n$ for ann $n\in\N$, $f(0)=0$, that $f$ is non-negative in a neighbourhood of $0$,  let
$$\alpha_m=\inf\left\{\int f\text{d}\mu:\mu\text{ is }F\text{-invariant}\right\}$$
and
$$\alpha_{M}=\sup\left\{\int f\text{d}\mu:\mu\text{ is }F\text{-invariant}\right\}.$$
We define,
\[J(\alpha)= \left\{ x \in [0,1] : \lim_{n \to \infty} \frac{1}{n} \sum_{i=0}^{n-1} f(F^i x)  = \alpha  \right\}.\]
We stress that in \cite{jjop} the Hausdorff dimension of the level sets, $\dim J(\alpha)$, is found for a more general class of functions. The result is that if $\alpha \in [\alpha_m, \alpha_M] \setminus 0$   then
\[\dim J(\alpha)= \sup \left\{ \frac{h(\mu)}{\lambda(\mu)} : \mu \in \mathcal{M}_F, \int f \ d \mu = \alpha \right\} \]
and it is also shown that  $\dim J(0)=1$.

With these stronger assumptions on our function $f$ we can use our results from Theorem \ref{analytic} to say more about the function $\alpha\to\dim J(\alpha)$. It is well known that $F$ can be related to a countable EMR map $T$.
 $$n(x)=\left\{\begin{array}{lll}1&\text{ if }&x\in [t,1]\\
 \inf\{n:F^n(x)\in I\}+1&\text{ if }&x\notin [t,1]\end{array}\right.$$
  and $T(x)=F^{n(x)}(x)$. Note that $T$ is an EMR map and we have that $\Lambda=[0,1]\backslash A$ and since $A$ is a countable set $\dim\Lambda=1$. We can also calculate $s_{\infty}=\frac{\beta}{\beta+1}$ (indeed see the proof of \cite[Proposition 1]{sar}). We can define $\phi(x)=\sum_{i=0}^{n(x)-1} f(F^i x)$ and $\psi(x)=r^n(x)$ and note that $\phi(x)\in\mathcal{R}$ and $\psi(x)\in\mathcal{R}_1$. For $\alpha\in \R $ let
$$X_{\alpha}=\left\{x\in [0,1]\backslash A:\lim_{n\to\infty}\frac{S_n\phi(x)}{S_n\psi(x)}=\alpha\right\}.$$

\begin{prop}
$\dim J(\alpha)=\dim X_{\alpha}$.
\end{prop}
\begin{proof}
It is immediate that $J(\alpha)\subseteq X_{\alpha}$ for all $\alpha$ except for $\alpha=0$. We also have that $J(0)\subset X_0\cup A$ and since $A$ is a countable set it follows that $\dim X_0\geq\dim J(0)=1$. Thus $\dim X_{0}=\dim J(0)=1$.

For the case when $\alpha\neq 0$ we first note that this implies that $\alpha\in U$. Thus there exists $C>0$ such that for all $x\in X_{\alpha}$ we have that $\limsup_{n\to\infty}\frac{S_n\psi(x)}{n}<C$. We now let $x\in X_{\alpha}$ and note that $\lim_{n\to\infty}\frac{S_{n+1}\psi(x)}{S_n\psi(x)}=1$. Thus, if we let $r_n=S_n\psi(x)$ we have that for $r_n\leq k\leq r_{n+1}$
$$\frac{\sum_{i=0}^{k-1}f(F^ix)}{k}\leq \frac{\sum_{i=0}^{r_{n+1}} f(F^ix)}{r_{n+1}}\frac{r_{n+1}}{r_n}$$
and by taking limits as $n\to\infty$ we can see that $x\in J(\alpha)$ and so the proof is complete.
\end{proof}

This means that we can apply Theorem \ref{analytic} to obtain more information about the function $\alpha\to\dim J(\alpha)$.

\begin{teo}\label{MPmap}
If $\alpha_m\leq 0=f(0)\leq \alpha_*=\int f\text{d}\mu_{SRB}<\alpha_M$ then
\begin{enumerate}
\item
The function $\alpha\to\dim J(\alpha)$ is analytic on $(\alpha_m,0)$ and $(\alpha_*,\alpha_M)$.
\item
$\dim J(\alpha)=1$ for $\alpha\in [0,\alpha_*]$
\item
If $\alpha_m<0$ then $\alpha\to\dim J(\alpha)$ is discontinuous at $\alpha=0$.
\end{enumerate}
\end{teo}

\begin{proof}
To prove this Theorem we first note that the potentials $\phi,\psi$ satisfy the assumptions for Theorem \ref{analytic}. Note that the absolutely continuous measure for $T$ projects to the absolutely continuous measure for $F$. Thus, in Theorem \ref{analytic} we have that $J_3=(\alpha_*,\alpha_M)$. We can determine that if $\lim_{n\to\infty}\frac{S_n\Psi(x)}{n}=\infty$ then $\lim_{n\to\infty}\frac{S_n\phi(x)}{S_n\psi(x)}=0$ and so $U=0$. Thus $J_2=(0,\alpha_*)$ and we can conclude that $J_3=(\alpha_m,0)$. The first two parts of the Theorem now immediately follow from Theorem \ref{analytic} and the final part follows since $s_{\infty}<1$ and thus the assumptions for Proposition \ref{discont} are met.
\end{proof}

\begin{rem}
We would be able to proof analogous results if $\alpha_*<0$ and $f$ is negative in a neighbourhood of $0$. The theorem would hold with weaker assumption on the function $f$. What we need is that $\phi$ is locally H\"{o}lder.
\end{rem}

\section{Multifractal Analysis for suspension flows}
Let $T$ be an EMR map and $\tau:(0,1] \to \R$ a positive function in $\mathcal{R}_{\eta}$. We consider the space
\begin{equation*}
Y:= \{ (x,t)\in (0,1]  \times \R \colon 0 \le t \le\tau(x)\},
\end{equation*}
with the points $(x,\tau(x))$ and $(T(x),0)$ identified for
each $x\in (0,1] $. The \emph{suspension semi-flow} over $T$
with \emph{roof function} $\tau$ is the semi-flow $\Phi = (
\varphi_t)_{t \ge 0}$ on $Y$ defined by
\[
 \varphi_t(x,s)= (x,
s+t) \ \text{whenever $s+t\in[0,\tau(x)]$.}
\]
In particular,
\[
 \varphi_{\tau(x)}(x,0)= (T(x),0).
\]
Because of the Markov structure of $T$ the flow $\Phi$ can be coded with a suspension semi-flow over a Markov shift defined on a countable alphabet.
Let $g:Y \to \R$ be  a potential and define
\begin{equation*}
K(\alpha):= \left\{(x,r) \in Y :\lim_{t \to \infty} \frac{1}{t} \int_0^t g(\varphi_s (x,r)) \text{d}s  =\alpha \right\}.
\end{equation*}
We define the \emph{Birkhoff spectrum} of $g$ by
\begin{equation*}
B(\alpha):= \dim (K(\alpha)).
\end{equation*}
It turns out that the results obtained to study multifractal analysis for quotients allow us to study the Birkhoff spectrum for flows.

\begin{rem}[Invariant measures]
We denote by $\M_\Phi$ the space of $\Phi$-invariant probability measures on $Y$. Recall that a measure $\mu$ on $Y$ is
\emph{$\Phi$-invariant} if $\mu(\varphi_t^{-1}A)= \mu(A)$ for every
$t \ge 0$ and every measurable set $A \subset Y$. Consider as well
the space $\M(T)$ of $T$-invariant probability measures on
$(0,1] $ and
\begin{equation*}
\M(T)(\tau):= \left\{ \mu \in \mathcal{M}(T): \int \tau \text{d} \mu < \infty \right\}.
\end{equation*}
Denote by $m$ the one dimensional Lebesgue measure and let $\mu \in \M(T)(\tau)$ then  it follows directly from classical results by Ambrose and Kakutani \cite{ak} that
\[(\mu \times m)|_{Y} /(\mu \times m)(Y) \in \M_{\Phi}.\]
Moreover,  If $\tau: (0,1] \to \R$ is bounded away from zero then there is a bijection between the spaces $\M_{\Phi}$ and $\M(T)(\tau)$.
\end{rem}

\begin{rem}[Kac's formula]
Given a continuous function $g \colon Y\to\R$ we define the function
$\Delta_g\colon (0,1] \to\R$~by
\[
\Delta_g(x)=\int_{0}^{\tau(x)} g(x,t) \, \text{d}t.
\]
The function $\Delta_g$ is also continuous, moreover
\begin{equation} \label{rela}
\int_{Y} g \, \text{d}R(\nu)= \frac{\int_\Sigma \Delta_g\, \text{d}
\nu}{\int_\Sigma\tau \, \text{d} \nu}.
\end{equation}
\end{rem}

\begin{rem}[Abramov's formula]
The entropy of a flow with respect to an invariant measure can be defined by  the entropy of the corresponding time one map.  Abramov \cite{a} and later Savchenko \cite{sav} proved that if
$\mu \in \M_{\Phi}$ is such that  $\mu=(\nu \times m)|_{Y} /(\nu \times m)(Y)$, where $\nu \in \M(T)$ then
\begin{equation}
h_{\Phi}(\mu)=\frac{h_{\sigma}(\nu)}{\int \tau \text{d} \nu}.
\end{equation}
\end{rem}
The following Lemma establishes a relation between the level sets determined by Birkhoff averages for the flow and the level sets of quotients for the map $T$.
\begin{lema} \label{igual1}
If
\begin{equation*}
\lim_{t \to \infty} \frac{1}{t} \int_0^t g(\varphi_s (x,r)) \text{d}s  =\alpha,
\end{equation*}
then
\begin{equation*}
\lim_{n \to \infty} \frac{\sum_{i=0}^{n} \Delta_g(T^{i} x)}{\sum_{i=0}^{n} \tau(T^{i} x)}= \alpha.
\end{equation*}
\end{lema}

\begin{proof}
Denote by
\[\tau_m(x):=\sum_{i=0}^{m-1}\tau(T^i x).\]
We have that
\begin{eqnarray*}
\int_0^{\tau_m(x)} g(\varphi_s(x,r)) \ d s= \sum_{i=0}^{m-1}\int_{\tau_i(x)}^{\tau_{i+1}(x)} g(\varphi_s(x,r) \ d s &=& \\ \sum_{i=0}^{m-1} \int_0^{\tau(T^ix)} g(\varphi_s(x,r) \ d s =
\sum_{i=0}^{m-1} \Delta_g(T^i x).
\end{eqnarray*}
In particular if
\begin{equation*}
\lim_{t \to \infty} \frac{1}{t} \int_0^t g(\varphi_s (x,r)) \text{d}s  =\alpha,
\end{equation*}
since $t \to \infty$ implies that $m \to \infty$,  we have that
\begin{equation*}
\lim_{m \to \infty} \frac{1}{\tau_m(x)} \sum_{i=0}^{m-1} \Delta_g(T^i(x)) =
\lim_{m \to \infty} \frac{\sum_{i=0}^{m-1} \Delta_g(T^i(x)) }{\sum_{i=0}^{m-1} \tau(T^i(x))}= \alpha.
\end{equation*}
\end{proof}

Let
\begin{equation*}
J(\alpha):= \left\{ x \in (0,1] :\lim_{n \to \infty} \frac{\sum_{i=0}^{n} \Delta_g(T^{i} x)}{\sum_{i=0}^{n} \tau(T^{i} x)} = \alpha \right\}.
\end{equation*}

Our main results establishes that we can compute the Hausdorff dimension of $K(\alpha)$ once we know the Hausdorff dimension of $J(\alpha)$.

\begin{teo} \label{mfa-flow}
Let $\alpha\in \R$ be such that $K(\alpha)\neq \emptyset$, then
\begin{equation*}
\dim K(\alpha) = \dim  J(\alpha) + 1.
\end{equation*}\end{teo}

We divide the proof of this results in a couple of Lemmas. The proof of the upper bound for the dimension of $K(\alpha)$ in terms of he dimension of $J(\alpha)$ is simpler.
\begin{lema}
Let $\alpha\in \R$ be such that $K(\alpha)\neq \emptyset$, then
\begin{equation*}
\dim K(\alpha) \leq \dim \left(J(\alpha) \times \R  \right)= \dim J(\alpha) + 1.
\end{equation*}
\end{lema}

\begin{proof}
First note that  if $(x,r) \in K(\alpha)$ then by virtue of Lemma \ref{igual1} we have that $x \in J(\alpha)$. Also if $(x,r) \in K(\alpha)$ then
 $(x,s) \in K(\alpha)$ for every $s \in [0, \tau(x))$. We therefore have
\begin{equation*}
K(\alpha) \subset \left\{(x,r) \in \R^2 : x \in J(\alpha) \text{ and } r \in [0, \tau(x)]   \right\}.
\end{equation*}
The box dimension of $\R$ and its Hausdorff dimension coincide, both are equal to one. Therefore, the dimension of the Cartesian product is the sum of the dimensions of each of the factors (see \cite[p.94]{fa}). The result now follows.
\end{proof}

In order to prove the lower bound we will use an approximation argument.

\begin{rem}[Compact setting] \label{cs}
Let $\C \subset Y$ is a compact $\Phi-$invariant set and consider the restriction of $g$ to the set $\C$ (which is a H\"older map). Then, it was proven by Barreira and Saussol \cite[Proposition 6]{bs2} that, $(x,r) \in K(\alpha)$ if and only if $x \in J(\alpha)$.
\end{rem}
The following Lemma completes the proof of Theorem \ref{mfa-flow}.
\begin{lema}
Let $\alpha\in \R$ be such that $K(\alpha)\neq \emptyset$, then
\begin{equation*}
\dim K(\alpha) \geq \dim \left(J(\alpha) \times \R  \right)= \dim J(\alpha) + 1.
\end{equation*}\end{lema}

\begin{proof}
Let $(\C_n)_n$ be an increasing sequence of a compact $\Phi-$invariant set that exhaust $Y$, denote by $C_n$ the projection of $\C_n$ onto $(0,1]$. We define
\begin{equation*}
K_n(\alpha):=K(\alpha) \cap \C_n \text{ and } J_n(\alpha):=J(\alpha) \cap C_n.\end{equation*}
By Remark \ref{cs} we have that
\begin{equation*}
K_n(\alpha)= \left\{ (x,r) \in \R^2: x \in J_n(\alpha) \text{ and } r \in [0, \tau(x)]  \right\}.
\end{equation*}
Hence, using the formula for the Hausdorff dimension of a Cartesian product (see \cite[p.94]{fa}), we have
\begin{equation*}
\dim K_n(\alpha) = \dim J_n(\alpha) + 1.
\end{equation*}
Therefore
\begin{equation*}
\lim_{n \to \infty } \dim K_n(\alpha) =\lim_{n \to \infty } \left( \dim J_n(\alpha) + 1\right) =
\dim J(\alpha) + 1 \leq \dim K(\alpha).
\end{equation*}
\end{proof}

It is a direct consequence of Theorem \ref{mfa-flow} that in order to describe the behaviour of the function $B(\alpha)$ we only need to understand $b(\alpha)$.
Recall that $\Lambda$ denotes the repeller for $T$. The following is a version of
Theorem \ref{analytic} in the suspension flow setting. It thus, describe the regularity properties of the map $B(\alpha)$.

\begin{teo}\label{analytic-flow}
Let $T$ be an EMR map, $\tau \in \R_{\eta}$ a roof a function and $\Phi$ the associated suspension semi-flow. Let $g: Y \to \R$ be a continuous potential such that $\Delta_g \in \mathcal{R}$. Assume that
\[\lim_{x\to 0}\frac{\tau(x)}{\log |T'(x)|}=\infty,\]
then there exist three pairwise disjoint intervals $J_1,J_2,J_3$ such that
\begin{enumerate}
\item
The domain of $K(\alpha)$ can be written as $J_1\cup J_2\cup J_3$,
\item
$J_1\leq J_2\leq  J_3$
\item
The function $\alpha\to\dim K_{\alpha}$ is analytic on $J_1$ and $J_3$
\item
For $\alpha\in J_2$, $\dim K_{\alpha}=\dim\Lambda + 1$
\item
It is possible that $J_1=\emptyset$, $J_3=\emptyset$ or that $J_2$ is a single point.
\end{enumerate}
\end{teo}

\section{Continued fractions}
In this section we consider examples involving the Gauss map and, hence, the continued fraction expansion of a number. Every irrational number $x \in [0,1]$ can be written in a unique way as a continued fraction,
\begin{equation*}
x = \textrm{ } \cfrac{1}{a_1 + \cfrac{1}{a_2 + \cfrac{1}{a_3 + \dots}}} = \textrm{ } [a_1 a_2 a_3 \dots],
\end{equation*}
where $a_i \in \mathbb{N}$. It well known that the Gauss map, $G:(0,1] \to (0,1]$, acts as the shift in the continued fraction expansion (see \cite[Chapter 3]{ew}), that is
\begin{equation*}
G([a_1 a_2  a_3 \dots ])= [a_2 a_3  a_4 \dots ].
\end{equation*}
\subsection{Arithmetic and Geometric averages}
Let $\psi:[0,1] \to \R$ be defined by $\psi([a_1 a_2  a_3 \dots ]):=a_1$. Note that the Birkhoff average of $G$ with respect to the potential $\psi$ is nothing but the arithmetic average of the digits in the continued fraction expansion
\[\lim_{n \to \infty} \frac{1}{n} \sum_{n=1}^{\infty} \psi(G^n(x)) =
\lim_{n \to \infty} \frac{a_1 + a_2 + \dots + a_n}{n} .\]
Note that \cite[p.83]{ew} for Lebesgue almost every point the arithmetic average is infinite.
The level sets induced by the arithmetic averages where studied in \cite{ij}. Consider now the function $\phi:[0,1] \to \R$ defined  by $\phi(x)= \log a_1$. The Birkhoff average of $G$ with respect to that function is the logarithm of the geometric average,
\[\lim_{n \to \infty} \frac{1}{n} \sum_{n=1}^{\infty} \phi(G^n(x)) =
\lim_{n \to \infty}  \log \sqrt[n]{a_1 a_2 \cdots a_n}.\]
For Lebesgue almost every point this sum takes the value \cite[p.83]{ew}
\[ \log \left( \prod_{n=1}^{\infty} \left(\frac{(n+1)^2}{n(n+2)}  \right)^{\log n / \log 2} \right).\]
The level sets determined by the geometric average were studied in \cite{flww,ks,pw}.
On the other hand, the Birkhoff average corresponding to $\log|G'(x)|$ is the \emph{Lyapunov exponent} of the point $x$, that is
\[\lambda(x):=\lim_{n  \to \infty} \frac{1}{n} \sum_{k=0}^{n-1} \log|G'(G^k(x))|\]
if this limit exists. This number  measures the exponential speed of approximation of an irrational number by its approximants, which are defined by,
\[\frac{p_n}{q_n}:=[a_1 \dots a_n].\]
That is (see \cite{pw})
\begin{equation*}
\left| x- \frac{p_n}{q_n} \right| \asymp \exp(-n\lambda(x)).
\end{equation*}
For Lebesgue almost every point this number equal to \cite[p.83]{ew},
\[  \frac{\pi^2}{6 \log 2}. \]
The multifractal analysis for this function has been studied in \cite{pw} and \cite{ks}.
The techniques developed in this paper allow us to study the following related level sets of the form
\begin{equation*}
J(\alpha):=\left\{ x \in [0,1] : \lim_{n\to \infty} \frac{\log (a_1 a_2 \cdots a_n)}{a_1 + a_2 + \dots + a_n} = \alpha \right\}.
\end{equation*}
Note that the quotient defining the level set is the quotient of the logarithm of the geometric average with  the arithmetic average. Indeed,
\[ \frac{\log (a_1 a_2 \cdots a_n)}{a_1 + a_2 + \dots + a_n} =\frac{\frac{1}{n}}{\frac{1}{n}} \frac{\log (a_1 a_2 \cdots a_n)}{a_1 + a_2 + \dots + a_n}=
\frac{ \log \sqrt[n]{a_1 a_2 \cdots a_n}}{\frac{a_1 + a_2 + \dots + a_n}{n}}. \]

\begin{lema}
We have that $\alpha_m=0$ and  $\alpha_M= \frac{\log 3}{3}$. Moreover, the set $J(0)$ has full Lebesgue measure.
\end{lema}

\begin{proof}
We have that $\alpha_m=0$. This is clear, just consider the number $x=[1 1 1 1 1 \dots ]$. It is easy to construct numbers belonging to $J(0)$ as the following example shows,  let $x=[1,2,2^2,2^3, \dots, 2^n \dots ]$. That is the number for which the digits in the continued fraction expansion are in geometric progression with ratio equal to $2$. Then
\[a_1 + a_2 + \dots + a_n= 1+2+2^2+2^3+ \dots+ 2^{n-1}= 2^n-1.\]
On the other hand
\[  \log (a_1 a_2 \cdots a_n) = \log(12^22^3 \cdots 2^{n-1})= \log 2^{1+2+3+\dots +(n-1)}= \frac{(n-1)n}{2} \log 2.    \]
Therefore
\[ \lim_{n \to \infty}  \frac{\log (a_1 a_2 \cdots a_n)}{a_1 + a_2 + \dots + a_n} =\lim_{n \to \infty} \frac{\frac{(n-1)n}{2} \log 2}{2^n-1}=0.\]
The fact that the set $J(0)$ has full Lebesgue measure is a direct consequence of the fact that for Lebesgue almost every point the arithmetic average is infinite and the geometric one is finite.

On the other hand we have that $\alpha_M= (\log 3)/3$. Indeed, it is well known that the arithmetic average is larger than the geometric one
\[ \frac{a_1 + a_2 + \dots + a_n}{n} \geq  \sqrt[n]{a_1 a_2 \cdots a_n}, \]
with equality if and only if $a_1=a_2=a_3= \dots =a_n$. Therefore the maximum of the quotient
\[ \frac{\log (a_1 a_2 \cdots a_n)}{a_1 + a_2 + \dots + a_n} \]
is achieved in an algebraic number of the form $x=[a,a,a, \dots ]$. In this case we obtain
\[ \frac{\log (a_1 a_2 \cdots a_n)}{a_1 + a_2 + \dots + a_n}=  \frac{\log a}{a}. \]
The maximum of the function $f(x)= (\log x)/x$ is attained at $x=e$. Since $a \in \N$ we have that the maximum is $(\log 3) /3$.
\end{proof}

\begin{lema}
We have that $\overline{\alpha}=\underline{\alpha}= 0$. In particular $U=[\alpha_m, \alpha_M]  \setminus (\underline{\alpha},\overline{\alpha}) =(0, (\log 3)/3)$
\end{lema}
\begin{proof}
Note that
\begin{equation*}
\lim_{x \to 0} \frac{\phi(x)}{\psi(x)}= \lim_{n \to \infty} \frac{\log n}{n}=0.
\end{equation*}
That is $\overline{\alpha}=\underline{\alpha}= 0$. Therefore $U=(\alpha_m, \alpha_M)=(0, (\log 3)/3)$.
\end{proof}
Therefore, a direct consequence of  Theorem \ref{main} is that

\begin{prop}
For every $\alpha \in (0, (\log 3)/3)$ we have that
\begin{equation*}
b(\alpha)=\dim J(\alpha)= \sup_{\mu\in \tilde{\mathcal{M}}(G)}\left\{\frac{h(\mu)}{\lambda(\mu)}: \frac{\int\phi\text{d}\mu}{\int\psi\text{d}\mu}=\alpha\right\}.
\end{equation*}
\end{prop}

\begin{lema}
We have that
\begin{equation*}
\lim_{x \to 0} \frac{\psi(x)}{\log|T'(x)|} = \infty.
\end{equation*}
\end{lema}
\begin{proof}
Note that if $ x \in (1/(n+1), 1/n)$ then $\psi(x)=n$ and $2\log n \geq \log|T'(x)|$. Thus,
\begin{equation*}
\lim_{x \to 0} \frac{\psi(x)}{\log|T'(x)|} \geq \lim_{n \to \infty} \frac{n}{2\log n}= \infty.
\end{equation*}
\end{proof}
In particular we have proved that the assumptions of Theorem \ref{analytic} are satisfied.

\subsection{Weighted arithmetic averages.}
Let $\rho,\gamma >0$ consider the level sets defined as the quotient of weighted arithmetic averages,
\begin{equation*}
J(\alpha):=\left\{ x \in [0,1] : \lim_{n\to \infty} \frac{a_1^{\gamma} + a_2^{\gamma} + \dots + a_n^{\gamma}}{a_1^{\rho} + a_2^{\rho} + \dots + a_n^{\rho}} = \alpha \right\}.
\end{equation*}
Let $\psi:[0,1] \to \R$ be defined by $\psi([a_1 a_2  a_3 \dots ]):=a_1^{\rho}$. Note that the Birkhoff average of $G$ with respect to the potential $\psi$ is nothing but the weighted arithmetic average of the digits in the continued fraction expansion
\[\lim_{n \to \infty} \frac{1}{n} \sum_{n=1}^{\infty} \psi(G^n(x)) =
\lim_{n \to \infty} \frac{a_1^{\rho} + a_2^{\rho} + \dots + a_n^{\rho}}{n} .\]
In an analogous way we define  $\phi:[0,1] \to \R$  by $\psi([a_1 a_2  a_3 \dots ]):=a_1^{\gamma}$.  The level sets determined by Birkhoff averages of  the potential $\psi$ where studied in \cite[Proposition 6.3]{ij}. In that context there exists essentially two different types of behaviour, depending if $\rho \in (0,1)$ or if $\rho \geq 1$.

\begin{rem} \label{al}
If $\gamma > \rho$ then
\[\overline{\alpha}=\underline{\alpha}=\lim_{x \to 0} \frac{\phi(x)}{\psi(x)} =\lim_{n \to \infty} n^{\rho-\gamma}=0.\]
Since the level sets are defined by the quotient of positive numbers we have that $\alpha_m=0$. Remark that if $x=[a,a,a, \dots]$ then
\begin{equation*}
 \lim_{n\to \infty} \frac{a^{\gamma} + a^{\gamma} + \dots + a^{\gamma}}{a^{\rho} + a^{\rho} + \dots + a^{\rho}} =a^{\rho- \gamma}.
\end{equation*}
In particular if $x=[1,1,1, \dots]$ we have that $\alpha=1 \in [\alpha_m, \alpha_m]$. On the other hand note that if $a \in \N$ then $a^{\rho} \leq a^{\gamma}$, therefore
\begin{equation*}
 \lim_{n\to \infty} \frac{a^{\gamma} + a^{\gamma} + \dots + a^{\gamma}}{a^{\rho} + a^{\rho} + \dots + a^{\rho}} \leq 1.
\end{equation*}
Thus, $\alpha_M=1$.
\end{rem}

\begin{prop}
If $\gamma > \rho$ then for every $\alpha \in (0, 1)$ we have that
$$b(\alpha):=\dim J(\alpha)=\sup_{\mu\in \tilde{\mathcal{M}}(T)}\left\{\frac{h(\mu)}{\lambda(\mu)}: \frac{\int\phi\text{d}\mu}{\int\psi\text{d}\mu}=\alpha\right\}$$
\end{prop}

\begin{proof}
In Remark \ref{al} we proved  that $\alpha_m=0$ and $E=\{0\}$. The result now follows by applying Theorem \ref{main}.
\end{proof}

\begin{lema}
We have that
\begin{equation*}
\lim_{x \to 0} \frac{\psi(x)}{\log|T'(x)|} = \infty.
\end{equation*}
\end{lema}
\begin{proof}
Note that if $ x \in [1/(n+1), 1/n]$ then $\psi(x)=n^{\gamma}$ and $2\log n \geq \log|T'(x)|$. Thus,
\begin{equation*}
\lim_{x \to 0} \frac{\psi(x)}{\log|T'(x)|} \geq \lim_{n \to \infty} \frac{n^{\gamma}}{2\log n}= \infty.
\end{equation*}
\end{proof}
In particular we have proved that the assumptions of Theorem \ref{analytic} are satisfied.

\end{document}